\documentclass[12pt]{amsart}
\usepackage{amsmath, amsthm, amssymb, amsfonts}
\usepackage{mathtools}
\usepackage{mathrsfs}
\usepackage{amscd, tikz-cd}
\usepackage{hyperref}
\usepackage[normalem]{ulem} % Preserve italic emphasis.
\usepackage{xcolor}
\usepackage[nounderscore]{syntax}
\usepackage{braket}
\usepackage{fullpage}
\usepackage{graphicx}%
\usepackage{multirow}%
\usepackage{amsmath,amssymb,amsfonts}%
\usepackage{amsthm}
\usepackage{mathrsfs}%
\usepackage[title]{appendix}%
\usepackage{xcolor}%
\usepackage{textcomp}%
\usepackage{manyfoot}%
\usepackage{booktabs}%
\usepackage{algorithm}%
\usepackage{algorithmicx}%
\usepackage{algpseudocode}%
\usepackage{listings}%
\usepackage{mathrsfs}
\usepackage{cleveref}
\usepackage[T1]{fontenc} %Not needed by LuaLaTeX or XeLaTeX
\usepackage[strings]{underscore}
\makeatletter
\newcommand*{\da@rightarrow}{\mathchar"0\hexnumber@\symAMSa 4B }
\newcommand*{\da@leftarrow}{\mathchar"0\hexnumber@\symAMSa 4C }
\makeatother % End the preceding internal-command definitions.
\newcommand{\Sp}{\operatorname{Sp}}
\newcommand{\GSp}{\operatorname{GSp}}
\usepackage{geometry}
\newtheorem{thm}{Theorem}[section]
\newtheorem{prop}[thm]{Proposition}
\newtheorem{lem}[thm]{Lemma}
\newtheorem{cor}[thm]{Corollary}

\theoremstyle{definition}

\theoremstyle{remark}
\newtheorem{remark}[thm]{Remark}

\numberwithin{equation}{section}

\usepackage{enumerate}
\usepackage{verbatim}
\usepackage{enumitem}

\usepackage[alphabetic]{amsrefs}

\newcommand{\F}{\mathbb{F}}
\newcommand{\Z}{\mathbb{Z}}
\newcommand{\Q}{\mathbb{Q}}

\newcommand{\Aut}{\operatorname{Aut}}

\newcommand{\End}{\operatorname{End}}
\newcommand{\Hom}{\operatorname{Hom}}

\DeclareMathOperator{\Gal}{Gal}

\usepackage{microtype}
\begin{document}

% Title
\title{Non-simple abelian varieties in a family: arithmetic approaches}
\author{Yu Fu}

\begin{abstract}
Inspired by the work of Ellenberg, Elsholtz, Hall, and Kowalski, we investigate how simplicity of the generic fiber of a one-parameter family of abelian varieties extends to other fibers. In \cite{EEHK09}, the authors studied a special case for certain families of Jacobians of curves using analytic methods and gave upper bounds on the number of fibers that are geometrically non-simple in terms of the height of the parameter. Using an arithmetic approach, we obtain analogous bounds for non-simple fibers over number fields in rational one-dimensional families of principally polarized abelian varieties whose geometric generic fiber has endomorphism ring \(\mathbb Z\) and whose relevant geometric monodromy orbits satisfy polynomial lower bounds. For families with full symplectic geometric monodromy, we obtain upper bounds with explicit logarithmic exponents independent of $g$.  Our method applies Heath-Brown type point counting bounds to certain covers with level structures and then optimizes the covers to derive the desired results.
\end{abstract}
%\subjclass[2020]{Primary: 14G05; Secondary: 11G10, 11G50}
\maketitle
\begin{sloppypar}
\section{Introduction}
Let \(\mathcal A_g\) denote the coarse moduli space of principally polarized abelian varieties of dimension \(g\), and let \(\mathcal A_g^*\) denote its minimal compactification. Let \(K\) be a number field, and let \(X\subset \mathcal A_g\) be an integral \(K\)-rational curve, which parameterizes a family \(A\to X\) of principally polarized abelian varieties. Let \(\overline X\) be the closure of \(X\) in \(\mathcal A_g^*\), and let \(\nu:\overline X^\nu\to\overline X\) be its normalization. Since \(X\) is \(K\)-rational, \(\overline X^\nu\simeq\mathbb P^1_K\). We fix such an identification and write \( \eta: \mathbb P^1_K\simeq\overline X^\nu \xrightarrow{\nu} \overline X \) for the normalization map. By abuse of notation, we also write \(\eta\) for its restriction \(\eta^{-1}(X)\to X\). For \(x\in X(K)\), let \(A_x\) denote the corresponding fiber. Let \(A_{K(X)} := A\times_X\operatorname{Spec}K(X)\) be the generic fiber of the family, and write \(A_{\overline{K(X)}}\) for its base change to an algebraic closure of \(K(X)\).

Throughout the paper, when we say that an abelian variety over a field $F$ is \textit{non-simple}, we mean it contains a nonzero proper abelian subvariety defined over $F$.

We ask the following question:

\medskip

\noindent\textbf{Question.} How many points $x\in X(K)$ of bounded height have non-simple fiber $A_x$?

\medskip

For $t\in \mathbb P^1(K)$, let $H_{\mathbb P^1}(t)$ be the multiplicative projective height. Define
\[S_\eta(B) := \left\{ t\in \mathbb P^1(K) \;: \eta(t)\in X, \; H_{\mathbb P^1}(t)\le B \text{ and } A_{\eta(t)} \text{ is non-simple over }K \right\}.\]
This set is finite by Northcott's theorem. The problem was studied by Ellenberg, Elsholtz, Hall, and Kowalski \cite{EEHK09} for certain one-parameter families of hyperelliptic Jacobians, namely the family of curves \[y^2=f(x)(x-t),\quad t\in \mathbb A^1,\] where $f\in \mathbb Z[x]$ is squarefree of degree $2g$ with \(g\ge2\). They proved the following logarithmic upper bound.

\begin{thm}[\cite{EEHK09}, Theorem~B]
For the preceding hyperelliptic family with \(K=\mathbb Q\), let \(J_t\) denote the Jacobian of \(y^2=f(x)(x-t)\). There exist constants $C_{f,g}\ge 0$ and $D_{f,g}\ge 1$, depending on $f$ and $g$, such that \[ \#\left\{t\in\mathbb Q: 
\begin{array}{l}
f(t)\ne0,\ H_{\mathbb P^1}(t)\le B,\\
J_t\text{ is not geometrically simple}
\end{array}
\right\} \le C_{f,g}\bigl(g^{2D_{f,g}}\log(2B)\bigr)^{11g^2} \] for all $B\ge 1$.
\end{thm}
Here the height is taken on the original parameter line $\mathbb A^1_{\Q}$, whose map to the moduli image need not be birational.
The constants in the preceding theorem are allowed to depend on \(f\). In the proof of \cite[Theorem~24]{EEHK09}, the finite field estimate of \cite[Proposition~22]{EEHK09} requires the reduction of \(f\) to remain squarefree of degree \(2g\). This condition may fail at primes dividing the leading coefficient or the discriminant of \(f\). At these finitely many primes and at the prime \(2\), one can instead use the trivial bound for the number of residue classes. Enlarging the sieve constant accordingly then introduces a dependence on \(f\).

The hyperelliptic family considered above has full symplectic geometric monodromy modulo every odd prime $\ell$; see \cite[proof of Theorem~8]{EEHK09}. Our goal is to prove an arithmetic analogue for rational one-dimensional families under polynomial lower bounds for the sizes of geometric monodromy orbits, without assuming full symplectic geometric monodromy or using an explicit equation for the family.

 A broader motivation for this problem comes from the study of Mumford--Tate groups in families of abelian varieties. A very general principally polarized abelian variety of dimension \(g\) has Mumford--Tate group \(\mathrm{GSp}_{2g}\), or equivalently has derived Hodge group \(\mathrm{Sp}_{2g}\). Special fibers with smaller Mumford--Tate group are expected to form a thin exceptional set; they often arise from extra Hodge classes, extra endomorphisms, or decompositions of the abelian variety. Thus non-simple fibers form a concrete and accessible part of this broader exceptional locus. Bounding such specializations is closely related in spirit to questions around André--Oort, unlikely intersections, and the distribution of special points in Shimura varieties.

To state the result, fix once and for all an embedding \( \iota=\iota_g:\mathcal A_g^*\hookrightarrow \mathbb P^N \) defined by sections of a sufficiently high power of the Hodge bundle on the minimal compactification $\mathcal A_g^*$; see Section~\ref{Heights of abelian varieties}. Throughout this paper, dependence on the fixed projective embeddings of the moduli spaces, including their chosen coordinates, is suppressed in the constants. We also write $\iota$ for its restriction to $\mathcal A_g$. For \(t\in\mathbb P^1(K)\), let \(P_t:=\iota(\eta(t)).\) Define
\[S_X(B) := \left\{ t\in \mathbb P^1(K) \;: \eta(t)\in X,\; H_{\mathbb P^N}(P_t)\le B \text{ and } A_{\eta(t)} \text{ is non-simple over }K \right\}.\]
Although the set \(S_X(B)\) is defined using \(\eta\), its cardinality is independent of this choice. Indeed, any other normalization parametrization is of the form \(\eta\circ\sigma\) with \(\sigma\in\operatorname{PGL}_2(K)\), and \(\sigma\) induces a bijection between the corresponding counting sets.

By functoriality of heights, a bound with a given logarithmic exponent for either counting function implies a bound with the same exponent for the other, with constants and thresholds allowed to depend on \(X\) and \(\eta\); see \eqref{eq:height-counting-inclusions}. We retain both counting functions to distinguish the two height conditions.

Our main result uses only polynomial lower bounds for the relevant geometric monodromy orbits.  Let \(\Gamma_\ell^{\operatorname{geom}}\) denote the geometric mod-\(\ell\) monodromy group, and let \(m_a:=\min(a,g-a)\) for \(1\le a\le g-1\). Let \(\kappa_0\) be the absolute exponent in Theorem~\ref{BCK}. One may take \(\kappa_0=12\). 

\begin{thm}\label{thm:polynomial-orbit-growth}
Let \(K\) be a number field of degree \(d_K\) over \(\mathbb Q\). Let \(X\subset \mathcal A_g\) be a rational curve over \(K\) parameterizing a one-dimensional family of principally polarized abelian varieties. Assume that \(g\ge2\) and that the geometric generic fiber satisfies \( \End_{\overline{K(X)}}(A_{\overline{K(X)}})=\mathbb Z. \) Suppose that there exist constants \(c_X>0\) and \(\ell_0\), and positive integers \(\alpha_{a,r}\) such that for every \(1\le a\le g-1\), \(0\le r\le m_a\), and every prime \(\ell\ge\ell_0\), each \(\Gamma_\ell^{\operatorname{geom}}\)-orbit on the set of subspaces \(V\subset A_{\overline{K(X)}}[\ell]\) with \(\dim_{\mathbb F_\ell}V=2a\) and \(\dim_{\mathbb F_\ell}(V\cap V^\perp)=2r\) has cardinality at least \(c_X\ell^{\alpha_{a,r}}\). Here orthogonality is taken with respect to the Weil pairing. Let \[\alpha_a^{\min}:=\min_{0\le r\le m_a}\alpha_{a,r}, \quad \beta_X:=\max_{1\le a\le g-1} \frac{8a(g-a)}{\alpha_a^{\min}}.\]
Then there exist constants \(C_{K,g,X}>0\) and \(B_{0,X}\) such that \[|S_X(B)| \le C_{K,g,X}(\log B)^{\kappa_0+\beta_X} \] for all \(B\ge B_{0,X}\).
Moreover, for every normalization parametrization \(\eta\), there exists a threshold \(B_{0,\eta,X}\) such that \[|S_\eta(B)| \le C_{K,g,X}(\log B)^{\kappa_0+\beta_X}\] for all \(B\ge B_{0,\eta,X}\).
Here \(C_{K,g,X}\) depends only on \(K\), \(g\), and \(X\), while \(B_{0,X}\) is allowed to depend on \(K\), \(g\), \(X\), the fixed projective embedding, and the orbit-growth threshold \(\ell_0\). The threshold \(B_{0,\eta,X}\) may additionally depend on the parametrization \(\eta\).

In particular, \(\beta_X\le8\lfloor g^2/4\rfloor\le2g^2\), so one may take \(12+8\lfloor g^2/4\rfloor\) as the logarithmic exponent in both estimates.
\end{thm}

For families with full symplectic geometric monodromy, the orbit sizes can be estimated uniformly in terms of \(g\). Theorem~\ref{thm:polynomial-orbit-growth} then gives the following bounds.

\begin{cor}\label{cor:full-monodromy}
Let \(K\) be a number field of degree \(d_K\) over \(\mathbb Q\). Let \(X\subset \mathcal A_g\) be a rational curve over \(K\) parameterizing a one-dimensional family of principally polarized abelian varieties, and assume that the geometric mod-\(\ell\) monodromy of the generic fiber is \( \operatorname{Sp}_{2g}(\mathbb F_\ell) \) for all sufficiently large primes \(\ell\). Then there exist a constant \(C_{K,g}\) and a threshold \(B_0(X)\) such that \[|S_X(B)| \le C_{K,g}(\log B)^{\kappa_0+4}\] for all \(B\ge B_0(X)\).
Moreover, for every normalization parametrization \(\eta\), there exists a threshold \(B_{0,\eta}\) such that \[|S_\eta(B)| \le C_{K,g}(\log B)^{\kappa_0+4}\] for all \(B\ge B_{0,\eta}\). 

Here \(C_{K,g}\) depends only on \(K\) and \(g\), while \(B_0(X)\) is allowed to depend on \(K\), \(g\), \(X\), the fixed projective embedding, and the monodromy threshold. The threshold \(B_{0,\eta}\) may additionally depend on the parametrization \(\eta\).

\end{cor}

\begin{remark}
The \(S_X(B)\) estimate can be made valid for all \(B\ge2\) by allowing its coefficient to depend on \(X\).
\end{remark}

The proof is different from the analytic sieve method of \cite{EEHK09}. We construct congruence covers of \(X\) which capture non-simple fibers, control the height change on these covers using Faltings' isogeny theorem, and then apply the sharp determinant-method bound of Binyamini--Cluckers--Kato to the resulting image curves.

In Section~\ref{s6}, we apply Theorem~\ref{thm:polynomial-orbit-growth} to the Mumford-type Jacobian family of \(C_{7,t}\) in \cite[Theorem~1.2(i)]{BHPS2025}. Its geometric monodromy is not the full symplectic group, so this application is not covered by Corollary~\ref{cor:full-monodromy}.

\subsection{Paper outline}

Section~\ref{s2} recalls the height comparison needed later, Faltings' isogeny estimate, the Binyamini--Cluckers--Kato point-counting theorem, and stabilizers of subspaces in symplectic spaces. In Section~\ref{3}, we construct the congruence covers, show how non-simple fibers lift to them, and describe their irreducible components. In Section~\ref{section change of heights}, Zarhin's trick gives a fixed projective ambient space, and we prove the height and degree bounds. Section~\ref{proof of the main theorem} combines these estimates with the Binyamini--Cluckers--Kato theorem to prove Theorem~\ref{thm:polynomial-orbit-growth} and Corollary~\ref{cor:full-monodromy}. Section~\ref{s6} applies the theorem to an explicit Mumford-type family.

\subsection{Acknowledgements}

The author thanks Jordan Ellenberg and Xiaoheng Wang for many helpful suggestions and discussions. The author is truly grateful for the invaluable suggestions provided by the anonymous referee. Sincere thanks for the time and effort they have devoted to reviewing this manuscript and helping bring this work to its best possible form.

\section{Preliminaries}\label{s2}

\subsection{Heights on abelian varieties}\label{Heights of abelian varieties}

Let \(\mathfrak A_g\) be the moduli stack of principally polarized abelian varieties of dimension \(g\), let \(\mathcal A_g\) be its coarse moduli space, and let \(\mathcal A_g^*\) be the minimal compactification. Let \(\mathcal L_g\) denote the Hodge line bundle on \(\mathfrak A_g\). Thus, if \(\pi:\mathcal X_g\to\mathfrak A_g\) is the universal abelian scheme and \(e:\mathfrak A_g\to\mathcal X_g\) is the zero section, then
\[\mathcal L_g=\det e^*\Omega^1_{\mathcal X_g/\mathfrak A_g} =\det \operatorname{Fil}^1\mathcal H^1_{dR}(\mathcal X_g/\mathfrak A_g).\]
By the construction of the minimal compactification in Faltings--Chai, after replacing \(\mathcal L_g\) by a sufficiently divisible positive tensor power, it descends to an ample line bundle on \(\mathcal A_g^*\), and a further positive tensor power is very ample \cite[Ch.~V, Theorem~2.3, p.~150]{FaltingsChai1990}.

We fix once and for all an integer \(w=w(g)\) and a \(K\)-vector space \(V_g\subset H^0(\mathcal A_g^*,\mathcal L_g^{\otimes w})\) giving an embedding
\[\iota=\iota_g:\mathcal A_g^*\hookrightarrow \mathbb P(V_g^\vee)=\mathbb P^N.\]
For a principally polarized abelian variety \(A/\overline{\mathbb Q}\), let \([A]\in\mathcal A_g(\overline{\mathbb Q})\) be its coarse moduli point and set
\[H_\iota(A):=H_{\mathbb P^N}(\iota([A])), \quad h_\iota(A):=\log H_\iota(A).\]
Changing the basis of \(V_g\) changes \(h_\iota\) by a bounded function, with the bound depending on the two chosen bases. Equivalently, this is the Weil height attached to $\mathcal{L}_g^{\otimes w}$ with the continuous metric pulled back from the standard projective metric on $\mathbb{P}^N$, up to \(O_{g,\iota}(1)\). From now on, we write \(h(A):=h_\iota(A)\).

We recall the comparison with the stable Faltings height. Let \(A/\overline{\mathbb Q}\) be an abelian variety, and choose a number field \(E\) over which \(A\) is defined. Faltings--Chai define the corresponding modular height \(h_{\mathrm{mod}}(A/E)\) as the normalized Arakelov degree of the Hodge bundle of the Néron model, equipped with the square-integral metric at the archimedean places \cite[Ch.~V, \S4, pp.~167--169]{FaltingsChai1990}. This height decreases under finite extension of \(E\) and stabilizes after passing to a finite extension \(E'/E\) over which \(A\) has semistable reduction. Such an extension exists by the semistable reduction theorem \cite[Ch.~I, Theorem~2.6, p.~9]{FaltingsChai1990}. We denote the resulting stable value by \(h_F(A)\). It is independent of the choices of \(E\) and \(E'\), and depends only on the \(\overline{\mathbb Q}\)-isomorphism class of \(A\).

Faltings--Chai show that the square-integral metric on the descended Hodge power has logarithmic singularities along the boundary \cite[Ch.~V, Proposition~4.5, pp.~168--169]{FaltingsChai1990}. Their argument for hermitian line bundles with logarithmic singularities \cite[Ch.~V, proof of Proposition~4.4, p.~168]{FaltingsChai1990} compares the height attached to this singular metric with the height attached to a continuous metric on the same ample line bundle for all algebraic points, with a logarithmic error in the latter height. Since \(\iota\) is defined by sections of the fixed power \(\mathcal L_g^{\otimes w}\), the height associated with the continuous projective metric is \(h(A)+O_{g,\iota}(1)\).

After semistable reduction, the height defined by the square-integral metric on \(\mathcal L_g^{\otimes w}\) is \(w h_F(A)\); see \cite[Ch.~V, \S4, p.~169]{FaltingsChai1990}. The compactification, the line bundle, the projective embedding, and the two metrics are fixed once and for all. Consequently, the comparison constants are uniform over all algebraic moduli points and do not depend on a field of definition of the point. After absorbing the fixed power \(w\) and the normalization convention into the constants, there exist constants \(a_\iota>0\) and \(C_\iota>0\), depending only on \(g\) and on the fixed embedding, such that
\begin{align}\label{ieq: 2.1}
\left|h_F(A)-a_\iota h(A)\right|
\le
C_\iota\log(h(A)+2)+C_\iota
\end{align}
for every principally polarized abelian variety \(A/\overline{\mathbb Q}\) of dimension \(g\).

Consequently, using the elementary estimate \(\log(x+2)\le \varepsilon x+C_\varepsilon\), after enlarging constants there is \(C_0\ge1\), depending only on \(g\) and \(\iota\), such that for every principally polarized abelian variety \(A/\overline{\mathbb Q}\) of dimension \(g\)
\begin{align}\label{ieq: 2.2}
h_F(A)\le C_0(h(A)+\log(h(A)+2)),
\end{align}
and
\begin{align}\label{ieq: 2.3}
h(A)\le C_0(h_F(A)+C_0).
\end{align}

We also need Faltings' isogeny estimate \cite[Lemma~5]{Faltings83}.

\begin{prop}[Faltings]\label{Faltings}
Let \(A\) and \(A'\) be abelian varieties over a number field, and let \( \varphi:A\to A' \) be an isogeny of degree \(m\). Then
\[|h_F(A')-h_F(A)|\le \frac12\log m.\]
\end{prop}

\begin{cor}\label{cor: comparison}
Let \(A\) and \(A'\) be principally polarized abelian varieties of dimension \(g\) over \(\overline{\mathbb Q}\), and suppose that there is an isogeny \( \varphi:A \to A' \) of degree \(m\). Then
\[h(A') \le C_{\mathrm{iso}} ( h(A)+\log m+\log(h(A)+2) ),\]
where \(C_{\mathrm{iso}}\) depends only on \(g\) and on the fixed embedding \(\iota\).
\end{cor}

\begin{proof}
Choose a number field \(E\) over which \(A\), \(A'\), and \(\varphi\) are defined. Proposition~\ref{Faltings} gives
\[h_F(A')\le h_F(A)+\frac12\log m.\]
Applying \((\ref{ieq: 2.3})\) to \(A'\), we get
\[h(A')\le C_0(h_F(A')+C_0) \le C_0(h_F(A)+\frac12\log m+C_0).\]
Using \((\ref{ieq: 2.2})\) for \(A\), and enlarging the constant, gives
\[h(A')\le C_{\mathrm{iso}} (h(A)+\log m+\log(h(A)+2)).\]
\end{proof}

\subsection{Uniform bound for rational points of bounded height on curves}

We shall use the following theorem of Binyamini--Cluckers--Kato \cite{BCK24}.

Let \(K\) be a number field with \( d_K=[K:\mathbb Q]. \) Let \(C\subset \mathbb P_K^n\) be an irreducible algebraic curve of degree \(d\), and let \(C(K,B)\) denote the set of \(K\)-rational points on \(C\) of absolute projective multiplicative height at most \(B\).

\begin{thm}[\cite{BCK24}, Theorem~1]\label{BCK}
For each \(n>1\), there is a constant \(c(K,n)\) and an absolute constant \(\kappa_0\) such that, for every irreducible curve \(C\subset \mathbb P_K^n\) of degree \(d\) and every \(B>2\),
\begin{align}\label{ieq: BCK}
\#C(K,B)
\le
c(K,n)d^2B^{\frac{2d_K}{d}}(\log B)^{\kappa_0}.
\end{align} One may take \(\kappa_0=12\).
\end{thm}

\subsection{Stabilizers of subspaces in symplectic spaces}
Let \(\ell\) be a prime and let \(E\) be a \(2g\)-dimensional symplectic vector space over \(\F_\ell\). Fix a symplectic basis to identify \(E\) with \(\F_\ell^{2g}\). For integers \(1\le a\le g-1\) and \(0\le r\le\min(a,g-a)\), let \(\mathscr S_{a,r}\) be the set of \(2a\)-dimensional subspaces \(V\subset E\) satisfying \( \dim_{\F_\ell}(V\cap V^\perp)=2r. \) Fix \(V_{a,r}\in\mathscr S_{a,r}\), define \(\mathcal H_{a,r}:=\operatorname{Stab}_{\GSp_{2g}(\F_\ell)}(V_{a,r})\), and write \(\mathcal H_{a,r}^{\operatorname{sp}}:=\mathcal H_{a,r}\cap\Sp_{2g}(\F_\ell)\).

\begin{lem}\label{lem:subspace-stabilizers}
The group \(\Sp(E)\) acts transitively on \(\mathscr S_{a,r}\). Moreover, if \(b:=g-a\), then \( \left| \Sp_{2g}(\F_\ell)/ \mathcal H_{a,r}^{\operatorname{sp}} \right| \asymp_g \ell^{e_{a,r}}, \) where \( e_{a,r} := 4ab-r(2r-1). \)
\end{lem}

\begin{proof}
Let \(V\in\mathscr S_{a,r}\), and let \(R:=V\cap V^\perp\). Then \(R\) is an isotropic subspace of dimension \(2r\), and \( R\subseteq V\subseteq R^\perp. \) The quotient \(R^\perp/R\) is symplectic, and \(V/R\) is a nondegenerate subspace of dimension \(2(a-r)\). Conversely, a choice of an isotropic \(2r\)-dimensional subspace \(R\subset E\), together with a nondegenerate \(2(a-r)\)-dimensional subspace of \(R^\perp/R\), determines \(V\). To prove transitivity, let \(V,V'\in\mathscr S_{a,r}\). Let \(R:=V\cap V^\perp\), and choose a complement \(S\) such that \(V=R\oplus S\). Since \(R\) is the radical of the restricted form on \(V\), the subspace \(S\) is nondegenerate. Choose a basis \(\{e_1,\ldots,e_{2r}\}\) of \(R\) and a symplectic basis \(\{e_{2r+1},\ldots,e_{a+r}, f_{2r+1},\ldots,f_{a+r}\}\) of \(S\). Since \(S\) is nondegenerate, \(E=S\perp S^\perp\), and the isotropic basis \(e_1,\ldots,e_{2r}\) of \(R\subset S^\perp\) extends to a symplectic basis \(\{e_1,\ldots,e_g,f_1,\ldots,f_g \}\) of \(E\). With respect to this basis,
\[V= \left\langle e_1,\ldots,e_{2r}, e_{2r+1},\ldots,e_{a+r}, f_{2r+1},\ldots,f_{a+r} \right\rangle.\]
The same construction gives an adapted symplectic basis \(\{e'_1,\ldots,e'_g,f'_1,\ldots,f'_g \}\) of \(E\) in which \[V'=\left\langle e'_1,\ldots,e'_{2r}, e'_{2r+1},\ldots,e'_{a+r}, f'_{2r+1},\ldots,f'_{a+r} \right\rangle.\] The linear map sending the first adapted symplectic basis to the second belongs to \(\Sp(E)\) and carries \(V\) onto \(V'\). Hence \(\Sp(E)\) acts transitively on \(\mathscr S_{a,r}\).

The choice of \(R\) contributes \( \left|\operatorname{IGr}(2r,E)\right| \asymp_g \ell^{m(2r,g)}, \) where \(\operatorname{IGr}(2r,E)\) denotes the Grassmannian of isotropic \(2r\)-dimensional subspaces of \(E\), and
\[m(2r,g) = 2r(2g-2r)-\frac{2r(2r-1)}2.\]
For fixed \(R\), the choice of the nondegenerate subspace \(V/R\) contributes
\[\left| \Sp_{2(g-2r)}(\F_\ell) \Big/ ( \Sp_{2(a-r)}(\F_\ell) \times \Sp_{2(b-r)}(\F_\ell) ) \right| \asymp_g \ell^{4(a-r)(b-r)}.\]
Therefore
\[e_{a,r} = m(2r,g)+4(a-r)(b-r) = 4ab-r(2r-1).\]
\end{proof}

\section{Congruence covers and their components}\label{3}

\subsection{Congruence covers}\label{Section 3.1}

The goal of this section is to study congruence covers of $X$ with level structures that allow us to capture non-simple specializations over $x \in X(K)$. We construct certain covers of $X$ by studying the level structure on the Siegel moduli space $\mathcal A_g$ induced by congruence subgroups of $\Gamma(1):=\Sp_{2g}(\Z)$. In particular, we work with the similitude level-$\ell$ cover \( q:\ \mathcal A_{g,\widetilde\Gamma(\ell)} \rightarrow \mathcal A_g, \) parametrizing principally polarized abelian varieties with a similitude level structure, namely a basis of the \(\ell\)-torsion respecting the Weil pairing up to a scalar. The group $\GSp_{2g}(\F_\ell)$ acts on the right. The cover is defined over \(K\), but after base change to \(\overline K\) it has \(\varphi(\ell)=\ell-1\) connected components, indexed by the multiplier. These components are defined over \(K(\zeta_\ell)\), where \(\zeta_\ell\) is a primitive \(\ell\)-th root of unity. For a subgroup \(\mathcal H\subset \GSp_{2g}(\F_\ell)\), the quotient \(\mathcal A_{g,\widetilde\Gamma(\ell)}/\mathcal H\) is geometrically connected precisely when \(\mathcal H\) has surjective multiplier. We show that every non-simple specialization lies on one of these subcovers after restricting to $X$.

Let $A_x$ be the fiber of the abelian scheme $A \to X$ over a $K$-rational point $x \in X(K)$, and let \(\ell\) be a prime. The $\ell$-torsion of $A_x$ carries a natural Galois action
\[\rho_{A_x,\ell} \colon  G_K \coloneq\operatorname{Gal}(\bar{K}/K) \rightarrow \operatorname{GSp}_{2g}(\F_\ell).\] Here the similitude multiplier equals the cyclotomic character:
\[\mu\big(\rho_{A_x,\ell}(\gamma)\big)\;=\;\chi_\ell(\gamma)\in \F_\ell^\times,\quad \gamma\in \Gal(\bar K/K).\]
Assume that $A_x$ is non-simple over $K$. Then there exists a nonzero proper abelian subvariety $B_x\subset A_x$ defined over $K$. We consider \(A_x[\ell]\) as an \(\mathbb F_\ell\)-vector space after the base change to \(\bar K\). Since $A_x$ is principally polarized, the $\ell$-torsion $A_x[\ell]$ carries the Weil pairing \( e_\ell:A_x[\ell]\times A_x[\ell]\to \mu_\ell. \)

Let $a=\dim B_x$. Let \(V:=B_x[\ell]\subset A_x[\ell],\) and let \(R:=V\cap V^\perp\) be the radical of \(V.\) Since \(V\) and \(V^\perp\) are \(G_K\)-stable, so is \(R\). The restriction of the Weil pairing to \(V\) is alternating, so its rank is even. Since \(\dim_{\F_\ell}V=2a\), the radical \(R\) has even dimension. We write \(\dim_{\F_\ell}R=2r.\) Let \(b:=g-a\). Then \( 0\le r\le\min(a,b). \) For each such pair \((a,r)\), use the fixed subspace \(V_{a,r}\subset\F_\ell^{2g}\) and its stabilizers \(\mathcal H_{a,r}\) and \(\mathcal H_{a,r}^{\operatorname{sp}}\) from Section~2.3.

 For any subgroup \(\mathcal H\subset\GSp_{2g}(\F_\ell)\) with surjective multiplier, the quotient \(\mathcal A_{g,\widetilde\Gamma(\ell)}/\mathcal H\) is defined over \(K\). We shall verify below that every \(\mathcal H_{a,r}\) satisfies this property. Let \[ Q_{\mathcal H_{a,r}}: \mathcal A_{g,\widetilde\Gamma(\ell)}/\mathcal H_{a,r} \rightarrow \mathcal A_g \] be the quotient level cover.
%Let $\widetilde X_{\mathcal H}$ be the lift of $X$ to \(Q_\mathcal{H}\),
%i.e. $\widetilde X_{\mathcal H}=Q_{\mathcal H}^{-1}(X),$ and let
%\(\pi:\widetilde X_{\mathcal H}\to X\) be the restriction of $Q_{\mathcal H}$ to $\widetilde X_{\mathcal H}$.
We conclude this subsection with the following proposition.

\begin{prop}\label{lift}
Fix a prime \(\ell\). Let \(x\in X(K)\) be a point such that \(A_x\) is non-simple over \(K\), and let \(B_x\subset A_x\) be a nonzero proper abelian subvariety defined over \(K\). Let \(a:=\dim B_x,\) \(V:=B_x[\ell],\) and \(\dim_{\F_\ell}(V\cap V^\perp)=2r. \) Then \(x\) lifts to a \(K\)-rational point of \( \mathcal A_{g,\widetilde\Gamma(\ell)}/\mathcal H_{a,r}. \) Moreover, \(\mathcal H_{a,r}\) has surjective multiplier.
\end{prop}

\begin{proof}
The subspace \(V\) is \(G_K\)-stable and has type \((a,r)\). Choose a similitude level structure \(\alpha_0:\F_\ell^{2g}\xrightarrow{\sim}A_x[\ell]. \) Then \(\alpha_0^{-1}(V)\subset\F_\ell^{2g}\) is a subspace of dimension \(2a\) whose restricted symplectic form has radical of dimension \(2r\). By Lemma~\ref{lem:subspace-stabilizers}, there exists \(g\in\Sp_{2g}(\F_\ell)\) such that \( g(V_{a,r})=\alpha_0^{-1}(V). \) Let \(\alpha=\alpha_0\circ g. \) Then \(\alpha\) is again a similitude level structure and satisfies \( \alpha(V_{a,r})=V. \)

The \(\mathcal H_{a,r}\)-orbit of \(\alpha\) determines a geometric point of \(Q_{\mathcal H_{a,r}}^{-1}(x)\). We now show that this point is defined over \(K\). For \(\sigma\in G_K\), let \( h_\sigma=\alpha^{-1}\circ\sigma(\alpha) \in\GSp_{2g}(\F_\ell). \) Since \(V\) is \(G_K\)-stable, we have \(h_\sigma(V_{a,r})= V_{a,r}. \) Thus \(h_\sigma\in\mathcal H_{a,r}\), and \(\sigma(\alpha)=\alpha\circ h_\sigma \) lies in the same \(\mathcal H_{a,r}\)-orbit as \(\alpha\). Therefore the point of \( \mathcal A_{g,\widetilde\Gamma(\ell)}/\mathcal H_{a,r} \) represented by the orbit of \(\alpha\) is fixed by \(G_K\), and hence defines a \(K\)-rational point above \(x\).

It remains to prove that the multiplier is surjective. Choose an adapted symplectic basis in which
\[V_{a,r} = \left\langle e_1,\ldots,e_{2r}, e_{2r+1},\ldots,e_{a+r}, f_{2r+1},\ldots,f_{a+r} \right\rangle.\]
For every \(\gamma\in\F_\ell^\times\), the element
\[
D_\gamma
=
\begin{pmatrix}
\gamma I_g&0\\
0&I_g
\end{pmatrix}
\]
preserves \(V_{a,r}\) and has multiplier \(\gamma\). Therefore \( \mu(\mathcal H_{a,r})=\F_\ell^\times. \)
\end{proof}
\begin{remark}
The case \(r=0\) recovers the block-diagonal subgroup. If \(r=\min(a,g-a)\), the subgroup \(\mathcal H_{a,r}\) is parabolic. For intermediate \(r\), it is the inverse image of a block-diagonal subgroup in the symplectic quotient \(R^\perp/R\) under the parabolic stabilizer of \(R=V_{a,r}\cap V_{a,r}^\perp\).
\end{remark}

\subsection{Moduli interpretations}\label{3.2}
Our next task is to bound the number of \(K\)-rational points of \(X\) that admit a \(K\)-rational lift to one of the quotient level covers attached to \( \mathcal H=\mathcal H_{a,r}, \) with \(1\le a\le g-1, \) and \(0\le r\le\min(a,g-a). \) In Section~\ref{section change of heights}, after passing to the normal projective completion of the pulled-back cover, we use this construction to define a morphism to a fixed projective space and to control the heights of lifted points.

The moduli interpretation of these covers is described as follows. For fixed \((a,r)\), the quotient \( \mathcal A_{g,\widetilde\Gamma(\ell)}/\mathcal H_{a,r} \) parametrizes principally polarized abelian varieties \((A,\lambda)\) together with a subspace \( V\subset A[\ell] \) such that \(\dim_{\F_\ell}V=2a,\) \(\dim_{\F_\ell}(V\cap V^\perp)=2r.\) In particular, the datum \((A, \lambda, V)\) determines the radical \(R:=V\cap V^\perp\) of \(V\) and the induced nondegenerate subspace \( V/R\subset R^\perp/R. \) In Section~\ref{section change of heights}, we apply Zarhin's construction to the quotient \(A/V\) and explain why it is functorial in families and descends to a morphism of moduli spaces.

\subsection{Components of the pullback covers}\label{integral}
In this subsection, we always assume that \(\End_{\overline{K(X)}}(A_{\overline{K(X)}})=\mathbb Z\).
Let \(X_{\mathrm{sm}}\subset X\) be the smooth locus of \(X\) over \(K\). Let \(U\subset\overline X^\nu\) be a nonempty open subscheme contained in \(\nu^{-1}(X_{\mathrm{sm}})\), chosen so that the pullback \( A_U:=A\times_X U \rightarrow U \) is an abelian scheme and \(\Aut(A_{\bar u},\lambda_{\bar u})=\{\pm1\}\) for every geometric point \(\bar u\) of \(U\), where \(\lambda_{\bar u}\) is the given principal polarization. Such an open subset exists because the geometric generic fiber has endomorphism ring \(\mathbb Z\). Since the scalar automorphisms \(\pm1\) fix every subspace of \(A_{\bar u}[\ell]\), the pulled-back quotient level cover attached to \(\mathcal H_{a,r}\) is the finite étale cover parametrizing subspaces of type \((a,r)\) in \(A_U[\ell]\). We choose \(U\) on \(\overline X^\nu\) independently of \(\ell\) and of its identification with \(\mathbb P^1_K\). Since the normalization map is an isomorphism over \(X_{\mathrm{sm}}\), we also regard \(U\) as an open subscheme of \(X\). Moreover, since \(\overline X^\nu\simeq\mathbb P^1_K\), the curve \(U\) is smooth and geometrically connected over \(K\).

The geometric fundamental group \(\pi_1(U_{\overline K})\) acts on the \(\ell\)-torsion of the geometric generic fiber. Since the family is principally polarized, after choosing a symplectic basis of the \(\ell\)-torsion this gives a representation
\[\rho^{\operatorname{geom}}_{\ell}: \pi_1(U_{\overline K}) \rightarrow \Sp_{2g}(\F_\ell).\]
Write \(\Gamma_\ell^{\operatorname{geom}}:=\operatorname{Im}(\rho_\ell^{\operatorname{geom}})\). We say that the family \(A_U\) over \(U\) has \textit{big geometric monodromy modulo} $\ell$ if the image of $\rho^{\operatorname{geom}}_{\ell}$ contains $\Sp_{2g}(\F_\ell)$; equivalently, after choosing a symplectic basis, the image is the full symplectic group. We say that the family has \textit{big geometric monodromy} if it has big geometric monodromy modulo $\ell$ for all sufficiently large primes $\ell$.

Similarly, for an abelian variety $A$ over a field $k$, we say that \emph{$A$ has big monodromy modulo $\ell$} if the image of its mod-$\ell$ Galois representation contains $\Sp_{2g}(\F_\ell)$ after choosing a symplectic basis. We say that $A$ has big monodromy if this holds for all sufficiently large primes $\ell$.

For any $g\geq 1$, there exists a constant $\ell_1(g)\geq 1$ such that if $\ell\geq \ell_1(g)$ and $A$ is an abelian variety of dimension $g$ over a field $k$ with big monodromy modulo $\ell$, then \(\End_{\bar k}(A)=\mathbb Z,\) and in particular $A$ is geometrically simple \cite[Proposition~4]{EEHK09}. However, the converse is not always true. For example, if $\mathcal A\to C$ is a family of abelian surfaces with real multiplication by $\mathbb Q(\sqrt d)$, then the mod-$\ell$ monodromy of the generic fiber can be contained in either \(\operatorname{GL}_2(\F_\ell)\times \operatorname{GL}_2(\F_\ell)\) or \(\operatorname{GL}_2(\F_{\ell^2}),\) while the generic fiber may still be geometrically simple.

For \(\mathcal H=\mathcal H_{a,r}\), let \( \widetilde U_{\mathcal H} := U\times_{\mathcal A_g} (\mathcal A_{g,\widetilde\Gamma(\ell)}/\mathcal H), \) where \(U\to\mathcal A_g\) is induced by the family. Then \(\widetilde U_{\mathcal H}\to U\) is a finite étale cover. For the general counting argument, this cover need not be geometrically connected, so we describe its irreducible components.

Since \(\mathcal H\) has surjective multiplier, the natural inclusion \(\Sp_{2g}(\F_\ell)\hookrightarrow\GSp_{2g}(\F_\ell)\) identifies \(\Sp_{2g}(\F_\ell)/\mathcal H^{\operatorname{sp}}\) with \(\GSp_{2g}(\F_\ell)/\mathcal H\). Let \(\mathcal O_1,\ldots,\mathcal O_s\) be the \(\Gamma_\ell^{\operatorname{geom}}\)-orbits on this set. The orbits \(\mathcal O_i\) describe the geometrically irreducible components of the pullback level cover over \(U_{\overline K}\). These components need not be defined over \(K\); the absolute Galois group \(G_K\) may permute them. We therefore group the \(\mathcal O_i\)'s into \(G_K\)-orbits \(\Omega_1,\ldots,\Omega_{\tau_{\mathcal H}}\).

For each \(\Omega_\alpha\), let \(Y_{\mathcal H,\alpha}\) be the normalization of \(\overline X^\nu\) in the function field of the corresponding \(K\)-irreducible component of the finite étale cover over \(U\). Let
\[\pi_{\mathcal H,\alpha}:Y_{\mathcal H,\alpha} \longrightarrow\overline X^\nu\]
be the resulting finite morphism. Then
\begin{equation}\label{eq:component-orbit-degree}
\deg(\pi_{\mathcal H,\alpha})
=\sum_{\mathcal O_i\in\Omega_\alpha}|\mathcal O_i|.
\end{equation}
Indeed, after base change to \(U_{\overline K}\), the geometric component corresponding to \(\mathcal O_i\) has degree \(|\mathcal O_i|\), and passing to the normal projective completion does not change the function-field degree. Consequently,
\begin{equation}\label{eq:total-cover-degree}
\sum_{\alpha=1}^{\tau_{\mathcal H}}\deg(\pi_{\mathcal H,\alpha})
=\deg(\widetilde U_{\mathcal H}\to U)
=\left|\Sp_{2g}(\F_\ell)/\mathcal H^{\operatorname{sp}}\right|.
\end{equation}

Let \(\widetilde X_{\mathcal H}:= \bigsqcup_{\alpha=1}^{\tau_{\mathcal H}}Y_{\mathcal H,\alpha}.\)
The maps \(\pi_{\mathcal H,\alpha}\) define a finite morphism \(\pi_{\mathcal H}:\widetilde X_{\mathcal H}\to\overline X^\nu\). Since \(\overline X^\nu\) is projective, each \(Y_{\mathcal H,\alpha}\) is an integral normal projective curve. Moreover, because \(U\) is normal and \(\widetilde U_{\mathcal H}\to U\) is finite étale, there is a canonical identification \(\pi_{\mathcal H}^{-1}(U)\simeq\widetilde U_{\mathcal H}.\)
We write
\[q_{\mathcal H}:=\eta\circ\pi_{\mathcal H}: \widetilde X_{\mathcal H}\longrightarrow\overline X \subset\mathcal A_g^*.\]

\begin{prop}\label{lift is integral}
Assume that the family has big geometric monodromy. Let \( \mathcal H=\mathcal H_{a,r} \) for some \(1\le a\le g-1,\) and \(0\le r\le\min(a,g-a).\) Then, for all sufficiently large primes \(\ell\), the cover \( \widetilde U_{\mathcal H}\to U \) is geometrically connected, and hence geometrically integral.
\end{prop}

\begin{proof}
For all sufficiently large primes \(\ell\), the image of \(\rho^{\operatorname{geom}}_\ell\) is \(\Sp_{2g}(\F_\ell)\). By Lemma~\ref{lem:subspace-stabilizers}, this group acts transitively on \(\Sp_{2g}(\F_\ell)/\mathcal H^{\operatorname{sp}}\). The preceding description of the geometric components therefore shows that \(\widetilde U_{\mathcal H}\) is geometrically connected. Since it is finite étale over the smooth curve \(U\), it is smooth. A connected smooth curve over \(\overline K\) is irreducible, so \(\widetilde U_{\mathcal H}\) is geometrically integral.
\end{proof}

\begin{remark}
The arithmetic big monodromy condition is known for a large class of abelian varieties defined over finitely generated fields. Serre proved the classical number-field result (cf. \cite[2.2.8]{Ser84}, \cite[Theorem~3]{Ser85}), and its characteristic-zero finitely generated field analogue implies big monodromy when \(\End_{\bar k}(A)=\mathbb Z\) and \(\dim A=2,6\), or \(\dim A\) is odd. Another large class is given by abelian varieties of \textit{Hall type}, namely those with \(\End_{\bar K}(A)=\mathbb Z\) and semistable reduction of toric dimension one at some place. Hall proved big monodromy for such varieties over global fields \cite{Hal11}, and Arias-de-Reyna, Gajda, and Petersen extended this to all dimensions over arbitrary finitely generated infinite ground fields \cite{AGP13}. For a family that is not weakly isotrivial, arithmetic big monodromy implies full symplectic geometric monodromy for all sufficiently large primes; see \cite[Proposition~4.2]{AGP13}.
    \end{remark}

\section{Heights and degrees on congruence covers}\label{section change of heights}
In this section, we construct a projective morphism associated with each congruence cover and bound the projective heights of \(K\)-rational lifts in terms of the height of the base point and the level \(\ell\). The height bound is given in Proposition~\ref{bounding}; the degree comparison is also proved at the end of this section.

\vspace{1em}
\noindent \textbf{The projective construction.} We have proved that for a suitable choice of prime $\ell$, there are congruence covers attached to the subspace stabilizers \(\mathcal H_{a,r}\), such that every non-simple specialization lifts to one of these covers. We now associate to each such cover a morphism to a fixed projective space and establish the height bound needed for the point-counting argument. We begin with the following lemma.

\begin{lem}\label{isotropic dual}
Let $A$ be a principally polarized abelian variety. Let $I \subset A[\ell]$ be any finite subgroup scheme. Let $I^{\perp}$ be the orthogonal complement of $I$ with respect to the Weil pairing $e_{\ell}$. Then $A/I$ is naturally dual to $A/I^{\perp}$.
\end{lem}
\begin{proof}
Let $q:A\to A/I$ be the quotient isogeny. Since $I\subset A[\ell]$, the map $[\ell]:A\to A$ factors uniquely as \( [\ell]=f\circ q \) for an isogeny $f:A/I\to A$. Dualizing gives \( q^\vee\circ f^\vee=[\ell]_{A^\vee}. \) Thus $\ker f^\vee\subset A^\vee[\ell]$. By the compatibility of the dual exact sequence for \(q\) with the Weil pairing, \(\ker f^\vee\) is the annihilator of \(\ker q=I\) in \(A^\vee[\ell]\). Under the fixed principal polarization \(A\simeq A^\vee\), this annihilator is \(I^\perp\). Hence \(f^\vee\) induces an isomorphism \( A^\vee/I^\perp\xrightarrow{\sim}(A/I)^\vee. \) Using the principal polarization, we obtain \( (A/I)^\vee\simeq A/I^\perp. \)

The same argument applies to abelian schemes over a base on which \(\ell\) is invertible, and the resulting isomorphism is compatible with arbitrary base change.
\end{proof}

We now explain why the preceding construction is functorial in families and therefore defines a morphism of moduli spaces. Let \(T\) be a \(K\)-scheme, and let \((\mathscr A,\lambda,\mathscr V)\) be a \(T\)-family represented by the quotient moduli problem attached to \(\mathcal H=\mathcal H_{a,r}\). Thus \((\mathscr A,\lambda)\) is a principally polarized abelian scheme over \(T\), and \(\mathscr V\subset\mathscr A[\ell]\) is a finite étale subgroup scheme whose geometric fibers have type \((a,r)\).

Let \( q:\mathscr A\to \mathscr B:=\mathscr A/\mathscr V \) be the quotient isogeny. Since \(\mathscr V\subset\mathscr A[\ell]\), multiplication by \(\ell\) factors uniquely as \([\ell]_{\mathscr A}=f\circ q\) for a unique isogeny \( f:\mathscr B\to \mathscr A. \) The homomorphism
\[\lambda_{\mathscr B} := f^\vee\circ\lambda\circ f: \mathscr B \to \mathscr B^\vee\]
is a polarization. Since \(\deg(q)=\ell^{2a}\) and \(\deg([\ell]_{\mathscr A})=\ell^{2g},\) we have \(\deg(f)=\ell^{2(g-a)}\) and \(\deg(\lambda_{\mathscr B})=\ell^{4(g-a)}.\) In particular, this polarization degree depends only on \(g\), \(a\), and \(\ell\).

The relative form of Lemma~\ref{isotropic dual} gives a canonical base-change-compatible isomorphism \(\mathscr B^\vee \simeq \mathscr A/\mathscr V^\perp.\) So
\[(\mathscr B\times\mathscr B^\vee)^4 \simeq ( (\mathscr A/\mathscr V) \times (\mathscr A/\mathscr V^\perp) )^4.\]
Let \( n:=\deg(\lambda_{\mathscr B})=\ell^{4(g-a)}. \) Choose once and for all integers \(u_1,u_2,u_3,u_4\) such that
\[u_1^2+u_2^2+u_3^2+u_4^2\equiv-1\pmod n.\]
The quaternion construction associated with this choice equips \( \mathscr Z:=(\mathscr B\times\mathscr B^\vee)^4 \) with a principal polarization; see \cite[Ch.~IX, Lemma~1.1, p.~205]{MoretBailly1985}. This construction is compatible with arbitrary base change: the fixed quaternion matrix defines an integral endomorphism of \(\mathscr B^4\); the relevant kernels and graph subgroup schemes, as well as quotients by finite locally free subgroup schemes, commute with base change; and the descended polarization is uniquely characterized by its pullback along the quotient isogeny. Hence
\[(\mathscr A,\lambda,\mathscr V) \mapsto \mathscr Z\]
defines a morphism from the corresponding quotient moduli stack to \(\mathfrak A_{8g}\).

Moreover, if a full level structure is replaced by its translate under an element of \(\mathcal H_{a,r}\), the subgroup \(\mathscr V\) is unchanged because \(\mathcal H_{a,r}\) stabilizes \(V_{a,r}\). Therefore the entire construction is \(\mathcal H_{a,r}\)-invariant and descends to the quotient. Passing to coarse moduli spaces gives a morphism
\[F_{\mathcal H}: \mathcal A_{g,\widetilde\Gamma(\ell)}/\mathcal H \to \mathcal A_{8g}.\]
Restricting this morphism to the pulled-back cover gives \( F_{\mathcal H}: \widetilde U_{\mathcal H} \rightarrow \mathcal A_{8g}. \) Since each component of \(\widetilde X_{\mathcal H}\) is a normal curve and \(\mathcal A_{8g}^*\) is proper, this morphism extends uniquely to \( F_{\mathcal H}: \widetilde X_{\mathcal H}\rightarrow\mathcal A_{8g}^*. \) Define
\[\Phi_{\mathcal H}:=(q_{\mathcal H},F_{\mathcal H}): \widetilde X_{\mathcal H} \rightarrow \mathcal A_g^*\times\mathcal A_{8g}^*.\]
Recall the fixed embedding \(\iota:\mathcal A_g^*\hookrightarrow\mathbb P^N\) from Section~\ref{Heights of abelian varieties}. Fix similarly an embedding \(\iota_{8g}:\mathcal A_{8g}^*\hookrightarrow\mathbb P^M\) induced by a fixed sufficiently divisible power of the Hodge bundle. Let \(R:=(N+1)(M+1)-1\), and define \(\iota_{\mathcal H}':=(\iota\times\iota_{8g})\circ\Phi_{\mathcal H}\) and \(\iota_{\mathcal H}:=\operatorname{Segre}\circ\iota_{\mathcal H}'\). Write \(C_{\mathcal H}:=\iota_{\mathcal H}(\widetilde X_{\mathcal H}) \subset\mathbb P^R\). For each component \(Y_{\mathcal H,\alpha}\), define \(q_{\mathcal H,\alpha}:=q_{\mathcal H}|_{Y_{\mathcal H,\alpha}}\), \(F_{\mathcal H,\alpha}:=F_{\mathcal H}|_{Y_{\mathcal H,\alpha}}\), and \(\iota_{\mathcal H,\alpha}:=\iota_{\mathcal H}|_{Y_{\mathcal H,\alpha}}\). Let \(C_{\mathcal H,\alpha}:=\iota_{\mathcal H,\alpha}(Y_{\mathcal H,\alpha})\). Since \(Y_{\mathcal H,\alpha}\) is projective and integral, its image \(C_{\mathcal H,\alpha}\) is closed and \(K\)-irreducible. Moreover, the first coordinate \(\iota\circ q_{\mathcal H,\alpha}\) is nonconstant, so \(C_{\mathcal H,\alpha}\) is a projective curve. The maps above can be arranged into the following diagram:
\[
\begin{tikzcd}
\widetilde X_{\mathcal H}
\arrow[r, "\iota_{\mathcal H}'"]
\arrow[d, "q_{\mathcal H}", swap]
&
\mathbb P^N\times\mathbb P^M
\arrow[r, hook, "\operatorname{Segre}"]
&
\mathbb P^{(N+1)(M+1)-1}
\\
\overline X
\arrow[r, hook, "\iota|_{\overline X}"]
&
\mathbb P^N.
\end{tikzcd}
\]
We do not claim that $\iota_{\mathcal H}$ is an embedding. What we need for the point-counting argument is that its restriction to each \(Y_{\mathcal H,\alpha}\) is generically finite onto its image, with degree bounded independently of $\ell$. The following Lemma~\ref{lem:zarhin-generically-finite} provides this boundedness.

\begin{lem}\label{rank one}
Let $A$ be an abelian variety over a field \(k\) of characteristic different from \(\ell\) with $\operatorname{End}_{\overline k}(A)=\mathbb Z$. Let $J\subsetneq A[\ell]$ be a proper subgroup scheme. Let $B_J=A/J$ and let \(q_J:A\to B_J \) be the quotient isogeny. Then
\[\operatorname{Hom}(A,B_J)=\mathbb Z\cdot q_J.\]
\end{lem}

\begin{proof}
It suffices to prove the statement after base change to \(\overline k\). Indeed, the natural map
\[\Hom_k(A,B_J) \rightarrow \Hom_{\overline k}(A_{\overline k},(B_J)_{\overline k})\]
is injective, and equality of morphisms can be checked after faithfully flat base change. Thus we may assume \(k=\bar k\), so that \(\End(A)=\mathbb Z\). Since \(q_J:A\to B_J\) is an isogeny, composition with \(q_J^{-1}\) in the rational Hom group gives
\[\Hom(A,B_J)\otimes_{\mathbb Z}\mathbb Q = \mathbb Q\cdot q_J.\]
Hence every \(f\in\Hom(A,B_J)\) can be written in \(\Hom(A,B_J)\otimes\mathbb Q\) as \( f=\frac ab q_J, \) with \(a,b\in\mathbb Z\), \((a,b)=1\), and \(b>0\). Equivalently, \( [b]_{B_J}\circ f=[a]_{B_J}\circ q_J. \)

We claim that $b=1$. Suppose a prime $p$ divides $b$. Then the morphism $[a]_{B_J}\circ q_J$ is divisible by $p$ in $\operatorname{Hom}(A,B_J)$; that is, it factors through multiplication by $p$ on $B_J$. Hence it kills all of $A[p]$. Since $(a,b)=1$, we have $p\nmid a$. Therefore multiplication by $a$ is an automorphism on $B_J[p]$, and
\[\ker([a]\circ q_J)\cap A[p]=\ker(q_J)\cap A[p]=J\cap A[p].\]
Thus divisibility by $p$ would force $A[p]\subseteq J.$ If $p\neq \ell$, this is impossible because $J\subset A[\ell]$ has no nontrivial $p$-torsion. If $p=\ell$, this would force $A[\ell]\subseteq J,$ contradicting the assumption that $J\subsetneq A[\ell]$ is proper. Hence no prime divides $b$, so $b=1$. Therefore $f\in \mathbb Z q_J$.
\end{proof}
\begin{lem}\label{ell-divisible-composition}
Let $A$ be an abelian variety over an algebraically closed field of characteristic different from $\ell$, with $\End(A)=\mathbb Z$. Let $I,J\subsetneq A[\ell]$ be proper subgroup schemes. Suppose \(I\neq J\). Let
\[u:A/I\to A/J,\quad v:A/J\to A/I\]
be homomorphisms. If \(v\circ u=[n]_{A/I}\), then \(\ell\mid n\).
\end{lem}

\begin{proof}
Let \(q_I:A\to A/I\) and \(q_J:A\to A/J\) be the quotient isogenies. We prove the contrapositive. Suppose that \(\ell\nmid n\). Then \(n\neq0\), so \(v\circ u=[n]_{A/I}\) is an isogeny. It follows that both \(u\) and \(v\) are isogenies. Since \(\deg(v\circ u)=n^{2g}\) is prime to \(\ell\), both \(\deg u\) and \(\deg v\) are prime to \(\ell\). By Lemma~\ref{rank one}, applied to $J$, we have \( \operatorname{Hom}(A,A/J)=\mathbb Z q_J. \) Therefore there exists an integer $m$ such that $u\circ q_I=m q_J.$ Taking degrees gives \( \deg(u)\cdot |I|=m^{2g}|J|. \) Since $\deg(u)$ is prime to $\ell$ and $I,J\subset A[\ell]$, this implies $v_\ell(|I|)\geq v_\ell(|J|).$ Applying the same argument to $v:A/J\to A/I$ gives $v_\ell(|J|)\geq v_\ell(|I|),$ hence $|I|=|J|.$ Then the degree identity also implies that $\ell\nmid m$. Now take \(x\in I\). Since \(q_I(x)=0\), the equality $u\circ q_I=mq_J$ gives $m q_J(x)=0.$ Because $\ell\nmid m$, multiplication by $m$ is an automorphism on $(A/J)[\ell]$. Hence $q_J(x)=0,$ so \(x\in J\). Thus $I\subseteq J.$ Since \(|I|=|J|\), we get \(I=J\), contradicting \(I\neq J\). Hence $\ell\mid n$.

\end{proof}

\begin{lem}\label{lem:zarhin-generically-finite}
Assume that the geometric generic fiber \(A=A_{\overline{K(X)}}\) satisfies \(\End(A)=\mathbb Z\). Let \(I,J\subset A[\ell]\) be proper subgroups of a fixed type \((a,r)\). For \(L\in\{I,J\}\), let \(B_L:=A/L\) and \(Z_L:=(B_L\times B_L^\vee)^4\). If $Z_I\simeq Z_J$ as abelian varieties, then \( I=J\) or \(I=J^\perp.\) Consequently, for each fixed subgroup type $\mathcal H$, the restriction of \(\Phi_{\mathcal H}\) to each \(K\)-irreducible component \(Y_{\mathcal H,\alpha}\) is generically finite onto its image, with generic degree at most $2$.
\end{lem}
\begin{proof}
Since $\operatorname{End}(A)=\mathbb Z$, every abelian variety isogenous to $A$ has endomorphism algebra $\mathbb Q$. Its endomorphism ring is an order in $\mathbb Q$, hence equals $\mathbb Z$. In particular, $\operatorname{End}(A/I)=\mathbb Z$ for every proper subgroup $I\subsetneq A[\ell]$ considered here.

Suppose \(Z_I\simeq Z_J\). Let \(F:Z_I\to Z_J\) be an isomorphism, and let \(G:Z_J\to Z_I\) be its inverse. By Lemma~\ref{isotropic dual}, we have \( (A/I)^\vee\simeq A/I^\perp \) and hence
\[Z_I\simeq (A/I)^4\times(A/I^\perp)^4, \quad Z_J\simeq (A/J)^4\times(A/J^\perp)^4.\]

Let \( \epsilon:A/I\to Z_I \) be the inclusion of the first copy of \(A/I\), and let \( \rho:Z_I\to A/I \) be the projection to the first copy of \(A/I\). For \(1\le i \le 4\), let
\[p_i:Z_J\to A/J, \quad j_i:A/J\to Z_J\]
be the projection onto and inclusion  of the \(i\)-th copy of \(A/J\), respectively. And for \(1\le s\le 4\), let
\[p_s^\perp:Z_J\to A/J^\perp, \quad j_s^\perp:A/J^\perp\to Z_J\]
be the projection onto and inclusion of the \(s\)-th copy of \(A/J^\perp\), respectively.

Define
\[f_i:=p_i\circ F\circ \epsilon:A/I\to A/J, \quad u_i:=\rho\circ G\circ j_i:A/J\to A/I,\]
and
\[\widetilde f_s:=p_s^\perp\circ F\circ \epsilon:A/I\to A/J^\perp, \quad v_s:=\rho\circ G\circ j_s^\perp:A/J^\perp\to A/I.\]
We have
\[F\circ\epsilon = \sum_{i=1}^4 j_i\circ f_i + \sum_{s=1}^4 j_s^\perp\circ \widetilde f_s,\]
where the sum is taken using the group law on \(Z_J\). Applying \(\rho\circ G\) and using \(G\circ F=\operatorname{id}_{Z_I}\), we obtain
\[\operatorname{id}_{A/I} = \rho\circ G\circ F\circ\epsilon = \sum_{i=1}^4 u_i\circ f_i + \sum_{s=1}^4 v_s\circ \widetilde f_s.\]
Each summand in the first sum is a composition
\[A/I\to A/J\to A/I,\]
and each summand in the second sum is a composition
\[A/I\to A/J^\perp\to A/I.\]
Since \(\End(A/I)=\mathbb Z\), each of the compositions \(u_i\circ f_i\) and \(v_s\circ\widetilde f_s\) is multiplication by an integer. If neither \(I=J\) nor \(I=J^\perp\), then Lemma~\ref{ell-divisible-composition} applied respectively to the pairs \((I,J)\) and \((I,J^\perp)\) shows that each of these integers is divisible by \(\ell\). Therefore the sum is also multiplication by an integer divisible by \(\ell\). But this sum is the identity endomorphism of \(A/I\), i.e. multiplication by \(1\), contradiction. Hence \(I=J \) or \(I=J^\perp\).

For a fixed type \((a,r)\), the subgroups \(I\) and \(J\) both have dimension \(2a\). If \(I=J^\perp\), then \( 2a=\dim J^\perp=2g-2a, \) so \(a=g-a\). These alternatives apply even when the two points lie on different geometric components of the pullback cover. Since the first coordinate of \(\Phi_{\mathcal H}\) is \(q_{\mathcal H}\), each geometric generic fiber over an image curve has at most two points. Thus the restriction to each \(Y_{\mathcal H,\alpha}\) has generic degree at most \(2\), and is generically injective if \(a\neq g-a\).
\end{proof}

For \(t\in U(K)\), write \(\widetilde P_t\) for a \(K\)-rational lift of \(t\) to \(\widetilde U_{\mathcal H}\), when one exists, and regard it as a point of \(\widetilde X_{\mathcal H}\). In the point-counting argument, we apply Theorem~\ref{BCK} to the image curves \(C_{\mathcal H,\alpha}\) and sum over the components. The bounded generic degrees of \(Y_{\mathcal H,\alpha}\to C_{\mathcal H,\alpha}\) will be used to compare the degrees of these image curves with the degrees of the corresponding components of the level cover.

\vspace{1em}
\noindent \textbf{Bounding the change of heights.} In what follows, the degree of the base curve is computed on its normalized projective model, which we have identified with \(\mathbb P^1_K\). Set \(d_X:=\deg_{\mathbb P^1_K}\eta^*\iota^*\mathcal O_{\mathbb P^N}(1)\). Thus \(d_X\) depends only on the curve \(X\) and the fixed projective embedding \(\iota\), and is independent of \(\ell\) and of the subgroup type \(\mathcal H_{a,r}\).

Since \(\iota\circ\eta\) is nonconstant, \(d_X\ge1\), and
\((\iota\circ\eta)^*\mathcal O_{\mathbb P^N}(1)
\simeq\mathcal O_{\mathbb P^1_K}(d_X)\).
By functoriality of heights, there exists a constant \(c_\eta\ge1\),
depending on \(\eta\) and the fixed embedding, such that
\begin{equation}\label{eq:normalization-height}
c_\eta^{-1}H_{\mathbb P^1}(t)^{d_X}
\le H_{\mathbb P^N}(P_t)
\le c_\eta H_{\mathbb P^1}(t)^{d_X}
\end{equation}
for every \(t\in\mathbb P^1(K)\). Consequently, for \(B\ge1\),
\begin{equation}\label{eq:height-counting-inclusions}
S_X(B)\subseteq S_\eta\bigl((c_\eta B)^{1/d_X}\bigr),
\qquad
S_\eta(B)\subseteq S_X(c_\eta B^{d_X}).
\end{equation}

\vspace{1em}
We first prove a logarithmic estimate valid over \(\overline K\).
\begin{lem}\label{lem:algebraic-point-height}
Fix a prime \(\ell\) and a subgroup type \(\mathcal H=\mathcal H_{a,r}\). For \(P\in\widetilde U_{\mathcal H}(\overline K)\), let \(A_P\) denote the fiber of the original family over \(\pi_{\mathcal H}(P)\in U(\overline K)\), and let \(A'_P\) denote the principally polarized abelian variety represented by \(F_{\mathcal H}(P)\in\mathcal A_{8g}(\overline K)\).

There exists a constant \(C_{\mathrm{pt}}>0\), depending only on \(g\) and on the fixed projective embeddings, such that
\begin{align}\label{eq: lem 4.5}
h_{\iota_{8g}}(A'_P) \le C_{\mathrm{pt}} \left( h_\iota(A_P) + \log(h_\iota(A_P)+2) + \log\ell + 1 \right)\end{align}
for every \(P\in\widetilde U_{\mathcal H}(\overline K)\).
\end{lem}

\begin{proof}
Choose a finite extension \(E/K\) over which \(P\), the two abelian varieties \(A_P,A'_P\), and the quotient construction are defined. By construction, there is an \(E\)-isogeny \(\varphi_P:A_P^8\to A'_P\) of degree \(\ell^{8g}\). Applying Corollary~\ref{cor: comparison} in dimension \(8g\), we obtain
\begin{align}\label{eq:4.2}
h_{\iota_{8g}}(A'_P)\le C_{\mathrm{iso}}\left(h_{\iota_{8g}}(A_P^8)+8g\log\ell+\log(h_{\iota_{8g}}(A_P^8)+2)\right),
\end{align}
where \(C_{\mathrm{iso}}\) depends only on \(g\) and on the fixed embedding \(\iota_{8g}\).

Consider the product morphism \(\mathcal A_g\to\mathcal A_{8g}\), \(A\mapsto A^8\), where \(A^8\) is equipped with the product principal polarization.  Since this morphism is fixed, functoriality of projective heights gives a constant \(C_1\ge2\), depending only on the fixed embeddings \(\iota\) and \(\iota_{8g}\), such that \(h_{\iota_{8g}}(A_P^8)\le C_1h_\iota(A_P)+C_1\). Consequently, \[\log(h_{\iota_{8g}}(A_P^8)+2)\le\log C_1+\log(h_\iota(A_P)+2).\] Substituting these two estimates into \eqref{eq:4.2} and choosing \(C_{\mathrm{pt}}\) sufficiently large in terms of \(g\), \(C_{\mathrm{iso}}\), and \(C_1\), we obtain \eqref{eq: lem 4.5}.

The projective heights used above are absolute heights, and the height comparisons in dimensions \(g\) and \(8g\) hold on \(\mathcal A_g(\overline{\mathbb Q})\) and \(\mathcal A_{8g}(\overline{\mathbb Q})\), respectively, with constants depending only on the dimensions and the fixed projective embeddings. Moreover, the stable Faltings height depends only on the \(\overline{\mathbb Q}\)-isomorphism class of the abelian variety. Therefore the same comparison constants apply regardless of the finite extension \(E/K\) over which the abelian varieties and the isogeny are defined. Thus \(C_{\mathrm{pt}}\) depends only on \(g\) and the fixed projective embeddings and is independent of \(E\), \(P\), \(\ell\), and \(\mathcal H\).
\end{proof}

We now specialize the preceding estimate to the \(K\)-rational points occurring in the counting problem.
\begin{prop}\label{bounding}
Fix a prime \(\ell\), and let \(\mathcal H=\mathcal H_{a,r}\). Let \(t\in U(K)\subset\mathbb P^1(K)\) be a point such that \(t\in S_\eta(B)\) or \(t\in S_X(B),\) and let \(\widetilde P_t\in\widetilde U_{\mathcal H}(K)\) be a lift satisfying \(\pi_{\mathcal H}(\widetilde P_t)=t.\) Then there exists an integer \(C_{\mathrm{ht}}\ge1\), depending only on \(g\) and on the fixed projective embeddings, such that
\[
H(\iota_{\mathcal H}(\widetilde P_t))
\le
\begin{cases}
C_{\eta,g}\,
\ell^{C_{\mathrm{ht}}}
B^{C_{\mathrm{ht}}d_X}
(\log B)^{C_{\mathrm{ht}}},
&
\text{if }t\in S_\eta(B),\\[4pt]
C_g\,
\ell^{C_{\mathrm{ht}}}
B^{C_{\mathrm{ht}}}
(\log B)^{C_{\mathrm{ht}}},
&
\text{if }t\in S_X(B),
\end{cases}
\]
for \(B\ge e\). Here \(C_{\eta,g}\) depends only on \(\eta\), \(g\), and the fixed embeddings, while \(C_g\) depends only on \(g\) and the fixed embeddings.
\end{prop}

\begin{proof}
Let \(A_{\eta(t)}\) be the fiber represented by \(\eta(t)\), and let \(A'_{\eta(t)}\) be the principally polarized abelian variety represented by \(F_{\mathcal H}(\widetilde P_t)\). Recall that \(\iota_{\mathcal H}(\widetilde P_t)\) is obtained from \(\iota(\eta(t))\) and \(\iota_{8g}(F_{\mathcal H}(\widetilde P_t))\) by the Segre embedding. Hence
\[
H(\iota_{\mathcal H}(\widetilde P_t))\asymp_g H_\iota(A_{\eta(t)})H_{\iota_{8g}}(A'_{\eta(t)}).
\]
Applying \eqref{eq: lem 4.5} in Lemma~\ref{lem:algebraic-point-height} to \(P=\widetilde P_t\) and exponentiating, we obtain
\[
H_{\iota_{8g}}(A'_{\eta(t)})\ll_g \ell^{C_{\mathrm{pt}}}H_\iota(A_{\eta(t)})^{C_{\mathrm{pt}}}(h_\iota(A_{\eta(t)})+2)^{C_{\mathrm{pt}}}.
\]
Combining these two estimates and choosing an integer \(C_{\mathrm{ht}}\ge\lceil C_{\mathrm{pt}}\rceil+1\), we obtain
\begin{align}\label{eq:general-lift-height}
H(\iota_{\mathcal H}(\widetilde P_t))
&\ll_g \ell^{C_{\mathrm{ht}}}H_\iota(A_{\eta(t)})^{C_{\mathrm{ht}}}(h_\iota(A_{\eta(t)})+2)^{C_{\mathrm{ht}}}.
\end{align}
Here \(C_{\mathrm{ht}}\) and the implied constant depend only on \(g\) and the fixed projective embeddings.

Suppose first that \(t\in S_\eta(B)\). By \eqref{eq:normalization-height}, we have
\[
H_\iota(A_{\eta(t)})\le C_{\eta,g}H_{\mathbb P^1}(t)^{d_X}\le C_{\eta,g}B^{d_X},
\]
where \(C_{\eta,g}\) depends only on \(\eta\), \(g\), and the fixed embedding. Consequently, \(h_\iota(A_{\eta(t)})+2\le d_X\log B+\log C_{\eta,g}+2\ll_{\eta,g}d_X\log B\) for \(B\ge e\). Substituting these estimates into \eqref{eq:general-lift-height} and enlarging \(C_{\eta,g}\) to absorb the fixed power of \(d_X\) and the implied constants gives the first bound.

If \(t\in S_X(B)\), then by definition \(H_\iota(A_{\eta(t)})\le B\), so \(h_\iota(A_{\eta(t)})+2\le\log B+2\ll\log B\) for \(B\ge e\). Substituting these estimates into \eqref{eq:general-lift-height} gives the second bound, with constants depending only on \(g\) and the fixed projective embeddings.
\end{proof}

\vspace{1em}
\noindent \textbf{Degree comparison.}
Assume that the geometric generic fiber has endomorphism ring
\(\mathbb Z\). Fix \(\mathcal H\), \(\ell\), and a component
\(Y_{\mathcal H,\alpha}\).
We give the following degree comparison argument. Let
\(
\delta_{\mathcal H,\alpha}
:=
\deg(Y_{\mathcal H,\alpha}\to C_{\mathcal H,\alpha})
\)
be the generic degree of the map to its image. Since the Segre embedding
satisfies
\(
\operatorname{Segre}^*\mathcal O(1)=\mathcal O(1,1),
\)
we have
\[
\iota_{\mathcal H,\alpha}^*\mathcal O_{\mathbb P^R}(1)
=
q_{\mathcal H,\alpha}^*\iota^*\mathcal O_{\mathbb P^N}(1)
\otimes
F_{\mathcal H,\alpha}^*\iota_{8g}^*\mathcal O_{\mathbb P^M}(1).
\]
Taking degrees gives
\[
\deg\iota_{\mathcal H,\alpha}^*\mathcal O_{\mathbb P^R}(1)
=
\deg q_{\mathcal H,\alpha}^*\iota^*\mathcal O_{\mathbb P^N}(1)
+
\deg F_{\mathcal H,\alpha}^*\iota_{8g}^*
\mathcal O_{\mathbb P^M}(1).
\]
Since \(q_{\mathcal H,\alpha}=\eta\circ\pi_{\mathcal H,\alpha}\), we have
\[
\begin{aligned}
\deg
q_{\mathcal H,\alpha}^*
\iota^*\mathcal O_{\mathbb P^N}(1)
=
\deg
\pi_{\mathcal H,\alpha}^*
\eta^*\iota^*\mathcal O_{\mathbb P^N}(1)=
d_X\deg(\pi_{\mathcal H,\alpha}).
\end{aligned}
\]

We next bound the degree contribution of \( F_{\mathcal H,\alpha}^*\iota_{8g}^*
\mathcal O_{\mathbb P^M}(1)\). We use the
following elementary height--degree comparison on projective curves.

\begin{lem}\label{lem:height-degree-on-curve}
Let \(C\) be a normal projective curve over a number field \(K\), let \(L\) and \(M\) be line bundles on \(C\), and let \(V\subset C\) be a nonempty open subscheme. Suppose that there exist an integer \(N\ge1\) and a constant \(\beta\in\mathbb R\) such that
\[h_L(P) \le N h_M(P)+\beta\]
for every \(P\in V(\overline K)\), where \(h_L\) and \(h_M\) are Weil heights associated with \(L\) and \(M\), respectively. Then
\[\deg L\le N\deg M.\]
\end{lem}

\begin{proof}
Let \(D:=M^{\otimes N}\otimes L^{-1}\). By additivity of heights, \(h_D=Nh_M-h_L+O(1)\), so the hypothesis bounds \(h_D\) below on \(V(\overline K)\), and hence on \(C(\overline K)\), since \(C\setminus V\) consists of finitely many geometric points. If \(\deg D<0\), then \(D^{-1}\) is ample and its height \(h_{D^{-1}}=-h_D+O(1)\) is bounded above. Choose \(m\ge1\) such that \((D^{-1})^{\otimes m}\) is very ample. A general linear projection gives a finite morphism \(f:C\to\mathbb P^1\) with \(f^*\mathcal O(1)\simeq(D^{-1})^{\otimes m}\). Thus \(h_{\mathbb P^1}(f(P))=m h_{D^{-1}}(P)+O(1)\) would be bounded above on \(C(\overline K)\), contrary to the surjectivity of \(f\) and the unboundedness of height on \(\mathbb P^1(\overline K)\). Therefore \(\deg D\ge0\), which is \(\deg L\le N\deg M\).
\end{proof}

Let
\(
M_{\mathcal H,\alpha}
:=
q_{\mathcal H,\alpha}^*
\iota^*\mathcal O_{\mathbb P^N}(1)
\)
and
\(
L_{\mathcal H,\alpha}
:=
F_{\mathcal H,\alpha}^*
\iota_{8g}^*\mathcal O_{\mathbb P^M}(1)
\)
be line bundles on \(Y_{\mathcal H,\alpha}\). We choose the corresponding pullback
heights
\[
h_{M_{\mathcal H,\alpha}}(P)
:=
h_{\mathbb P^N}
(\iota(q_{\mathcal H,\alpha}(P)))
\]
and
\[
h_{L_{\mathcal H,\alpha}}(P)
:=
h_{\mathbb P^M}
(\iota_{8g}(F_{\mathcal H,\alpha}(P))).
\]
For every
\(P\in\pi_{\mathcal H,\alpha}^{-1}(U)(\overline K)\),
we have
\(
h_{M_{\mathcal H,\alpha}}(P)=h_\iota(A_P)
\)
and
\(
h_{L_{\mathcal H,\alpha}}(P)=h_{\iota_{8g}}(A'_P).
\)
Then by Lemma~\ref{lem:algebraic-point-height},
\[
h_{L_{\mathcal H,\alpha}}(P)
\le
C_{\mathrm{pt}}
\left(
h_{M_{\mathcal H,\alpha}}(P)
+
\log(h_{M_{\mathcal H,\alpha}}(P)+2)
+
\log\ell
+
1
\right).
\]
Since \(h_{M_{\mathcal H,\alpha}}(P)\ge0\) and
\[
\log(h_{M_{\mathcal H,\alpha}}(P)+2)
\le
h_{M_{\mathcal H,\alpha}}(P)+\log2,
\]
we let
\(
N_g
:=
\left\lceil2C_{\mathrm{pt}}\right\rceil
\)
and
\(
\beta_{g,\ell}
:=
C_{\mathrm{pt}}
(\log\ell+1+\log2).
\)
Then
\[
h_{L_{\mathcal H,\alpha}}(P)
\le
N_g h_{M_{\mathcal H,\alpha}}(P)+\beta_{g,\ell}
\]
for every
\(P\in\pi_{\mathcal H,\alpha}^{-1}(U)(\overline K)\). Therefore
Lemma~\ref{lem:height-degree-on-curve} gives
\[
\deg L_{\mathcal H,\alpha}
\le
N_g\deg M_{\mathcal H,\alpha}.
\]

Since \(L_{\mathcal H,\alpha}\) is the pullback of
\(\mathcal O_{\mathbb P^M}(1)\) under a morphism, we also have
\(
\deg L_{\mathcal H,\alpha}\ge0.
\)
Combining the preceding estimates gives
\[
d_X\deg(\pi_{\mathcal H,\alpha})
\le
\deg
\iota_{\mathcal H,\alpha}^*\mathcal O_{\mathbb P^R}(1)
\le
(1+N_g)d_X\deg(\pi_{\mathcal H,\alpha}).
\]
Since
\[
\deg
\iota_{\mathcal H,\alpha}^*\mathcal O_{\mathbb P^R}(1)
=
\delta_{\mathcal H,\alpha}\deg C_{\mathcal H,\alpha}
\]
and
\(
1\le\delta_{\mathcal H,\alpha}\le2
\)
by Lemma~\ref{lem:zarhin-generically-finite}, it follows that
\begin{align}\label{degree-CH}
\deg C_{\mathcal H,\alpha}
\asymp_g
d_X\deg(\pi_{\mathcal H,\alpha}).
\end{align}

\section{Counting non-simple fibers}\label{proof of the main theorem}

\subsection{Counting points on the image curves}

For the counting argument, assume that \(g\ge2\) and that the geometric generic fiber has endomorphism ring \(\mathbb Z\). Write \(d_K:=[K:\mathbb Q]\), and use the notation of Sections~\ref{3} and~\ref{section change of heights}.

The finitely many points in \(\eta^{-1}(X)\setminus U\) contribute only \(O_X(1)\) points. We include this contribution when passing from the estimates over \(U\) to the full counting sets below. Hence, in the counting argument below, we can count only points lying over \(U\). After removing these finitely many points, we continue to use the notation \(S_\eta(B)\) and \(S_X(B)\) for the resulting subsets of \(U(K)\).

Define
\[\mathcal B_\eta(\ell,B) := C_{\eta,g}\ell^{C_{\mathrm{ht}}} B^{C_{\mathrm{ht}}d_X} (\log B)^{C_{\mathrm{ht}}},\]
and
\[\mathcal B_X(\ell,B) := C_g\ell^{C_{\mathrm{ht}}} B^{C_{\mathrm{ht}}} (\log B)^{C_{\mathrm{ht}}}.\]
These are the height bounds from Proposition~\ref{bounding}.

For \(1\le a\le g-1\), let \(S_{\eta,a}(B)\) be the set of points \(t\in S_\eta(B)\) for which \(A_{\eta(t)}\) contains a nonzero proper \(K\)-defined abelian subvariety of dimension \(a\). Define \(S_{X,a}(B)\) similarly. Recall that \(m_a=\min(a,g-a)\), and set
\[e_{a,r}:=4a(g-a)-r(2r-1),\quad e_a^{\max}:=\max_{0\le r\le m_a}e_{a,r}=4a(g-a).\]
For each \(a\), let \(\ell_a\) be an auxiliary prime, common to all \(r\)-types with this fixed \(a\), to be chosen below. Let \(S_{\eta,a,r}(B)\) be the subset of \(S_{\eta,a}(B)\) consisting of points admitting a \(K\)-rational lift to the cover attached to \(\mathcal H_{a,r}\) at the prime \(\ell_a\), and define \(S_{X,a,r}(B)\) similarly. By Proposition~\ref{lift},
\[S_{\eta,a}(B)\subseteq\bigcup_{r=0}^{m_a}S_{\eta,a,r}(B),\quad S_{X,a}(B)\subseteq\bigcup_{r=0}^{m_a}S_{X,a,r}(B).\]

Fix \((a,r)\), and write \(\mathcal H:=\mathcal H_{a,r}\). At level \(\ell_a\), abbreviate \(Y_{\mathcal H,\alpha}\), \(\pi_{\mathcal H,\alpha}\), and \(C_{\mathcal H,\alpha}\) to \(Y_\alpha\), \(\pi_\alpha\), and \(C_\alpha\), respectively, and let \(D_\alpha:=\deg C_\alpha\) and \(\tau:=\tau_{\mathcal H}\). The degree comparison \eqref{degree-CH} gives
\[D_\alpha\asymp_g d_X\deg(\pi_\alpha).\]
Over \(U\), a \(K\)-rational lift gives a \(K\)-rational point of one of the curves \(Y_\alpha\). For each base point under consideration, choose one such lift and count its image in \(C_\alpha(K)\). Since the first coordinate of \(\Phi_{\mathcal H}\) is \(q_{\mathcal H}\), distinct points of \(U\) give distinct projective images. Each \(C_\alpha\) is \(K\)-irreducible, so Theorem~\ref{BCK} gives
\begin{equation}\label{eq:counting-eta-components}
|S_{\eta,a,r}(B)|\le
c(K,R)\sum_{\alpha=1}^{\tau}D_\alpha^2
\mathcal B_\eta(\ell_a,B)^{2d_K/D_\alpha}
\bigl(\log\mathcal B_\eta(\ell_a,B)\bigr)^{\kappa_0}.
\end{equation}
The same argument gives
\begin{equation}\label{eq:counting-X-components}
|S_{X,a,r}(B)|\le
c(K,R)\sum_{\alpha=1}^{\tau}D_\alpha^2
\mathcal B_X(\ell_a,B)^{2d_K/D_\alpha}
\bigl(\log\mathcal B_X(\ell_a,B)\bigr)^{\kappa_0}.
\end{equation}
Here \(B\) is sufficiently large that the height bounds exceed \(2\).

By \eqref{eq:total-cover-degree} and Lemma~\ref{lem:subspace-stabilizers},
\[\sum_{\alpha=1}^{\tau}\deg(\pi_\alpha) =\left|\Sp_{2g}(\F_{\ell_a})/\mathcal H^{\operatorname{sp}}\right| \asymp_g\ell_a^{e_{a,r}}.\]
The upper bound \(D_\alpha\ll_g d_X\deg(\pi_\alpha)\) then gives
\begin{equation}\label{eq:degree-square-sum}
\begin{aligned}
\sum_{\alpha=1}^{\tau}D_\alpha^2
\ll_g d_X^2\sum_{\alpha=1}^{\tau}\deg(\pi_\alpha)^2
 \le d_X^2\left(\sum_{\alpha=1}^{\tau}\deg(\pi_\alpha)\right)^2
\ll_g d_X^2\ell_a^{2e_{a,r}}.
\end{aligned}
\end{equation}

\subsection{Polynomial orbit growth}

\begin{proof}[Proof of Theorem~\ref{thm:polynomial-orbit-growth}]
In this proof, dependence on \(X\) includes dependence on the fixed \(c_X\) and \(\alpha_{a,r}\) in the orbit hypothesis.
For each \(1\le a\le g-1\), choose a prime \(\ell_a\) common to all \(r\)-types with this fixed \(a\), such that
\[
(\log B)^{1/\alpha_a^{\min}} \le\ell_a\le 2(\log B)^{1/\alpha_a^{\min}}.
\]
For \(B\) sufficiently large, Bertrand's postulate allows these choices with \(\ell_a\ge\ell_0\) for every \(1\le a\le g-1\).

Fix \((a,r)\), and use the notation above. Lemma~\ref{lem:subspace-stabilizers} identifies the subspaces in the theorem with \(\Sp_{2g}(\F_{\ell_a})/\mathcal H_{a,r}^{\operatorname{sp}}\).
By the orbit hypothesis and \eqref{eq:component-orbit-degree},
\[
\deg(\pi_\alpha)\ge c_X\ell_a^{\alpha_{a,r}}.
\]
Using \eqref{degree-CH}, we obtain
\[
D_\alpha\asymp_g d_X\deg(\pi_\alpha)
\gg_{g,X}d_X\ell_a^{\alpha_{a,r}}.
\]
With the choice of \(\ell_a\), this gives
\[
D_\alpha\gg_{g,X} d_X(\log B)^{\alpha_{a,r}/\alpha_a^{\min}}.
\]
Since \(\alpha_{a,r}\ge\alpha_a^{\min}\), we conclude that \(D_\alpha\gg_{g,X}d_X\log B\).

The exponential factor in \eqref{eq:counting-X-components} is bounded.
Indeed, since \(\alpha_a^{\min}\ge1\), our choice of \(\ell_a\) gives \(\log\ell_a\le\log2+\log\log B\le\log B\) for \(B\ge e\). Hence
\[
\begin{aligned}
\log\mathcal B_X(\ell_a,B)
&=\log C_g+C_{\mathrm{ht}}\log\ell_a
+C_{\mathrm{ht}}\log B+C_{\mathrm{ht}}\log\log B \ll_g\log B.
\end{aligned}
\]
Since \(D_\alpha\gg_{g,X}d_X\log B\) and \(d_X\ge1\), the ratio \(\frac{\log\mathcal B_X(\ell_a,B)}{D_\alpha}\) is \(O_{g,X}(1)\).
Thus
\[
\mathcal B_X(\ell_a,B)^{2d_K/D_\alpha}=O_{K,g,X}(1).
\]
The same estimate also gives
\[
(\log\mathcal B_X(\ell_a,B))^{\kappa_0} \ll_g(\log B)^{\kappa_0}.
\]
Therefore \eqref{eq:counting-X-components} gives
\[ |S_{X,a,r}(B)|\ll_{K,g,X} \sum_{\alpha=1}^{\tau}D_\alpha^2(\log B)^{\kappa_0}.
\]
Consequently, \eqref{eq:degree-square-sum} gives
\[
\begin{aligned} |S_{X,a,r}(B)| &\ll_{K,g,X}d_X^2\ell_a^{2e_{a,r}}(\log B)^{\kappa_0} \ll_{K,g,X}d_X^2 (\log B)^{\kappa_0+2e_{a,r}/\alpha_a^{\min}}. \end{aligned}
\]
Since \(\frac{2e_{a,r}}{\alpha_a^{\min}} \le\frac{2e_a^{\max}}{\alpha_a^{\min}}\le\beta_X\), summing over \((a,r)\) and absorbing \(d_X^2\) into the constant, we obtain constants \(A_{K,g,X}>0\) and \(B_{0,X}\) such that
\[|S_X(B)|\le A_{K,g,X}(\log B)^{\kappa_0+\beta_X} \] for \(B\ge B_{0,X}\). After enlarging the constant to absorb the finitely many points outside \(U\), this estimate holds for the full set \(S_X(B)\).
The height bound \(\mathcal B_X(\ell_a,B)\), the degree estimates, the choices of the auxiliary primes, and the number of excluded points are all independent of \(\eta\). Thus both \(A_{K,g,X}\) and \(B_{0,X}\) can be chosen independently of \(\eta\).

The second inclusion in \eqref{eq:height-counting-inclusions} gives
\[
\begin{aligned}
|S_\eta(B)| &\le A_{K,g,X} \bigl(d_X\log B+\log c_\eta\bigr)^{\kappa_0+\beta_X} \le A_{K,g,X}(d_X+1)^{\kappa_0+\beta_X} (\log B)^{\kappa_0+\beta_X}
\end{aligned}
\]
for \(B\ge B_{0,\eta,X}:=\max\{e,c_\eta,B_{0,X}\}\).
Taking \(C_{K,g,X}:=A_{K,g,X}(d_X+1)^{\kappa_0+\beta_X}\) proves both estimates with the same coefficient, while the threshold for \(S_\eta(B)\) may depend on \(\eta\).

Finally, since \(\alpha_a^{\min}\ge1\),
\[
\beta_X\le\max_{1\le a\le g-1}8a(g-a) =8\left\lfloor\frac{g^2}{4}\right\rfloor\le2g^2,
\]
which gives the stated bound for the logarithmic exponent.
\end{proof}

\subsection{Full symplectic monodromy}

\begin{proof}[Proof of Corollary~\ref{cor:full-monodromy}]
For \(g=1\), the statement is trivial, since elliptic curves are simple. Hence we assume \(g\ge2\). Full symplectic geometric monodromy implies that the geometric generic fiber has endomorphism ring \(\mathbb Z\), by \cite[Proposition~4]{EEHK09}. 
By Lemma~\ref{lem:subspace-stabilizers}, the geometric monodromy acts transitively on the subspaces of each type \((a,r)\), and the size of this orbit is \(\asymp_g\ell^{e_{a,r}}\). We may therefore take \(\alpha_{a,r}=e_{a,r}\) and \(c_X=c_g>0\), with \(c_g\) depending only on \(g\).

Fix \(a\), and let \(m_a:=\min(a,g-a)\) and \(n_a:=\max(a,g-a)\). The exponent \(e_{a,r}\) decreases with \(r\), so
\[
e_a^{\max}=4m_an_a,\quad e_a^{\min}:=e_{a,m_a}=m_a(4n_a-2m_a+1).
\]
Thus \(\alpha_a^{\min}=e_a^{\min}\), and
\[\frac{2e_{a,r}}{e_a^{\min}} \le\frac{2e_a^{\max}}{e_a^{\min}} =\frac{8n_a}{4n_a-2m_a+1}<4. \]

Let
\[
\gamma_g:=\max_{1\le a\le g-1}
\frac{2e_a^{\max}}{e_a^{\min}}<4,
\qquad \delta_g:=4-\gamma_g>0.
\]
With \(\alpha_{a,r}=e_{a,r}\), we have \(\beta_X=\gamma_g\).

Theorem~\ref{thm:polynomial-orbit-growth} therefore gives a constant \(A_{K,g,X}>0\) such that
\[
\begin{aligned}
|S_X(B)| &\le A_{K,g,X}(\log B)^{\kappa_0+\gamma_g} &&\text{for }B\ge B_{0,X},\\
|S_\eta(B)| &\le A_{K,g,X}(\log B)^{\kappa_0+\gamma_g} &&\text{for }B\ge B_{0,\eta,X}.
\end{aligned}
\]
Let \[ B_*(X):=\exp\left(\max\{1,A_{K,g,X}\}^{1/\delta_g}\right). \]
For \(B\ge B_*(X)\), we have \(A_{K,g,X}\le(\log B)^{\delta_g}\). Thus, taking
\[ B_0(X):=\max\{B_{0,X},B_*(X)\}, \qquad B_{0,\eta}:=\max\{B_{0,\eta,X},B_*(X)\}, \]
gives
\[
\begin{aligned}
|S_X(B)|
&\le(\log B)^{\kappa_0+4}
&&\text{for }B\ge B_0(X),\\
|S_\eta(B)|
&\le(\log B)^{\kappa_0+4}
&&\text{for }B\ge B_{0,\eta}.
\end{aligned}
\]
Since \(B_{0,X}\) and \(B_*(X)\) are independent of \(\eta\), these thresholds have the stated dependencies. Taking a common constant \(C_{K,g}\ge1\) proves both estimates.
\end{proof}

\section{A Mumford-type family}\label{s6}

We give an example to which Theorem~\ref{thm:polynomial-orbit-growth} applies but Corollary~\ref{cor:full-monodromy} does not. Let \(C_{7,t}\) be the hyperelliptic family in \cite[Theorem~1.2(i), equation~(1.1)]{BHPS2025}, and let \(U_0:=\mathbb P^1_{\mathbb Q}\setminus\{0,1,\infty\},\) \(\mathcal J_t:=\operatorname{Jac}(C_{7,t}).\) The Jacobians of the family form a principally polarized abelian scheme \(\mathcal J\to U_0\) of relative dimension \(4\). We first check that the map \(t\mapsto[\mathcal J_t]\) is birational onto its image in \(\mathcal A_4\).

\begin{lem}\label{lem:bhps-birational-parameter}
Let \(X\subset\mathcal A_4\) be the reduced closure of the moduli image of \(\mathcal J\to U_0\). The moduli map \(t \mapsto[\mathcal J_t]\) is birational onto \(X\). Therefore the normalization of its closure \(\overline X\subset\mathcal A_4^*\) is \(\mathbb P^1_{\mathbb Q}\).
\end{lem}

\begin{proof}
Write the equation in \cite[equation~(1.1)]{BHPS2025} as \(y^2=t f_t(x,1)\), where \(f_t(x,z)\) is the degree ten homogenization of the polynomial inside parentheses. To recover \(t\) from the isomorphism class of \(C_{7,t}\), we construct two invariants of the binary form \(f_t\) using transvectants. For binary forms \(p,q\) of degrees \(m,n\), respectively, define
\[(p,q)_s:= \frac{(m-s)!(n-s)!}{m!n!} \sum_{i=0}^s(-1)^i\binom{s}{i} \frac{\partial^s p}{\partial x^{s-i}\partial z^i} \frac{\partial^s q}{\partial x^i\partial z^{s-i}}, \quad 0\le s\le\min(m,n).\]
For a binary form \(f\) of degree ten, define
\[u_f:=((f,f)_6,f)_8, \quad v_f:=((f,f)_4,f)_8, \quad w_f:=(((f,f)_2,f)_7,f)_9,\]
and
\[I_6(f):=(u_f,u_f)_2, \quad I_8(f):=(v_fw_f,f)_{10}.\]
The invariants \(I_6\) and \(I_8\) have coefficient degrees \(6,8\) and determinant weights \(30,40\), respectively. Thus \(I_6^4/I_8^3\) is unchanged by invertible linear changes of variables and nonzero scalar multiplication of \(f\). Direct calculation in \(\mathbb Q[t]\) gives
\[
\begin{aligned}
I_6(f_t)=-\frac{2^3 3^{10}7^2}{5^4}t^{12}(t-1)^4,\quad
I_8(f_t)=\frac{2^{10}3^{11}7^5}{5^8 11^2 13}t^{16}(t-1)^5.
\end{aligned}
\]
Therefore we get
\begin{align}\label{eq:bhps-parameter-recovery}
t=1+\frac{2^{18}7^7}{3^7 5^8 11^6 13^3}
\frac{I_6(f_t)^4}{I_8(f_t)^3}.
\end{align}
The ratio \(I_6(f_t)^4/I_8(f_t)^3\) is unchanged when \(f_t\) is replaced by \(t f_t\). Over \(\overline{\mathbb Q}\), the degree-two map from a smooth hyperelliptic curve of genus at least \(2\) to \(\mathbb P^1\) is unique up to an automorphism of \(\mathbb P^1\); see \cite[proof of Theorem~12.1(b)]{Milne1986Jacobian}. Hence an isomorphism between two such curves induces an automorphism of \(\mathbb P^1\) carrying one branch divisor to the other. The corresponding binary forms therefore differ by an invertible linear change of coordinates and a nonzero scalar. Thus the ratio above depends only on the geometric isomorphism class of \(C_{7,t}\).

Let \(t_1,t_2\in U_0(\overline{\mathbb Q})\) lie in the open set where \(I_8(f_t)\ne0\). If \(C_{7,t_1}\simeq C_{7,t_2}\), then
\[\frac{I_6(f_{t_1})^4}{I_8(f_{t_1})^3} =\frac{I_6(f_{t_2})^4}{I_8(f_{t_2})^3}.\]
Applying \eqref{eq:bhps-parameter-recovery} to \(t_1\) and \(t_2\) gives \(t_1=t_2\). By Torelli's theorem \cite[Theorem~12.1(a)]{Milne1986Jacobian}, an isomorphism between the canonically principally polarized Jacobians of \(C_{7,t_1}\) and \(C_{7,t_2}\) implies an isomorphism of the curves, and hence also \(t_1=t_2\). The moduli map is therefore generically injective and birational onto its image.
\end{proof}

We next prove two facts about the finite groups that occur in the monodromy calculation.

\begin{lem}\label{lem:ell-generated-orbits}
Let \(E\) be a finite-dimensional vector space over \(\mathbb F_\ell\), and let \(G\subseteq\operatorname{GL}(E)\) act irreducibly on \(E\). Suppose that \(G\) is generated by elements of order \(\ell\). Then every nonzero proper subspace \(V\subset E\) has an orbit \(G\cdot V:=\{\gamma(V):\gamma\in G\}\) of cardinality at least \(\ell\).
\end{lem}

\begin{proof}
If \(|G\cdot V|=s<\ell\), the action of $G$ on this orbit gives a homomorphism \(G\to  S_s\), where \( S_s\) is the symmetric group on \(s\) letters. Since \(S_s\) has no element of order \(\ell\), every generator acts trivially. Hence \(G\) fixes \(V\), contrary to irreducibility.
\end{proof}

\begin{lem}\label{lem:mumford-residual-tensor}
Let \(\ell\) be a prime, let \(k=\overline{\mathbb F}_\ell\), and let \(W:=(k^2)^{\otimes3}\). The representations
\[
\begin{aligned}
\rho_1:\operatorname{SL}_2(\mathbb F_\ell)^3
&\longrightarrow\operatorname{GL}(W),\\
(g_1,g_2,g_3)&\longmapsto g_1\otimes g_2\otimes g_3,
\end{aligned}
\]
and
\[
\begin{aligned}
\rho_2:\operatorname{SL}_2(\mathbb F_{\ell^3})
&\longrightarrow\operatorname{GL}(W),\\
g&\longmapsto g\otimes g^{(\ell)}\otimes g^{(\ell^2)}
\end{aligned}
\]
are irreducible. Here \(g^{(\ell^i)}\) is obtained by raising each matrix entry to its \(\ell^i\)-th power.
\end{lem}

\begin{proof}
Let \(I\) be the \(2\times2\) identity matrix, and let
\[
E_+:=\begin{pmatrix}0&1\\0&0\end{pmatrix},
\quad
E_-:=\begin{pmatrix}0&0\\1&0\end{pmatrix}.
\]

We first consider \(\rho_1\). The matrices \(I+E_+\) and \(I+E_-\) in each factor generate \(\operatorname{Mat}_2(k)\) as a \(k\)-algebra. It follows that the \(k\)-algebra generated by the image of \(\rho_1\) contains \(M\otimes I\otimes I,\) \(I\otimes N\otimes I,\) and \(I\otimes I\otimes P\) for all \(M,N,P\in\operatorname{Mat}_2(k)\). Hence the image of \(\rho_1\) generates \(\operatorname{End}_k(W)\simeq\operatorname{Mat}_8(k).\) Any \(\rho_1\)-invariant subspace is therefore stable under every endomorphism of \(W\), so it is zero or \(W\). So \(\rho_1\) is irreducible.

We now consider \(\rho_2\). Let \(W'\subseteq W\) be a \(\rho_2\)-invariant \(k\)-subspace, and let \(\pi:W\longrightarrow W/W'\) be the quotient map. For \(t\in\mathbb F_{\ell^3}\), let \(u(t):=I+tE_+\). Then
\[
\begin{aligned}
\rho_2(u(t))
&=(I+tE_+)\otimes(I+t^\ell E_+)\otimes(I+t^{\ell^2}E_+)\\
&=
\sum_{\epsilon_0,\epsilon_1,\epsilon_2\in\{0,1\}}
t^{\epsilon_0+\epsilon_1\ell+\epsilon_2\ell^2}
E_+^{\epsilon_0}\otimes E_+^{\epsilon_1}\otimes E_+^{\epsilon_2},
\end{aligned}
\]
where \(E_+^0=I\).

Fix \(w\in W'\). Since \(W'\) is invariant under \(\rho_2\), we have \(\pi(\rho_2(u(t))w)=0 \) for every \(t\in\mathbb F_{\ell^3}.\) Let \(T\) be an indeterminate and define the polynomial with coefficients in \(W/W'\) by
\[F_w(T):= \sum_{\epsilon_0,\epsilon_1,\epsilon_2\in\{0,1\}} T^{\epsilon_0+\epsilon_1\ell+\epsilon_2\ell^2} \pi\!\left( ( E_+^{\epsilon_0}\otimes E_+^{\epsilon_1}\otimes E_+^{\epsilon_2} )w \right).\]
The preceding equality says that \(F_w(t)=0\) for every \(t\in\mathbb F_{\ell^3}\).

Choose a \(k\)-basis \(b_1,\ldots,b_d\) of \(W/W'\). For each \(\epsilon=(\epsilon_0,\epsilon_1,\epsilon_2)\in\{0,1\}^3\), write
\[\pi\!\left(\bigl(E_+^{\epsilon_0}\otimes E_+^{\epsilon_1}\otimes E_+^{\epsilon_2}\bigr)w\right)=\sum_{j=1}^d c_{j,\epsilon}b_j, \quad c_{j,\epsilon}\in k.\]
Substituting these expressions into \(F_w(T)\), we obtain
\[F_w(T)=\sum_{j=1}^d f_j(T)b_j, \quad f_j(T):= \sum_{\epsilon\in\{0,1\}^3} c_{j,\epsilon} T^{\epsilon_0+\epsilon_1\ell+\epsilon_2\ell^2}\in k[T].\]
For every \(t\in\mathbb F_{\ell^3}\), we have \(0=F_w(t)=\sum_{j=1}^d f_j(t)b_j.\) The linear independence of the \(b_j\) therefore gives \(f_j(t)=0\) for every \(j\) and every \(t\in\mathbb F_{\ell^3}\). Each \(f_j(T)\) has degree at most \(1+\ell+\ell^2<\ell^3\), so these \(\ell^3\) distinct roots force \(f_j(T)\) to be the zero polynomial.

Since the eight exponents \(\epsilon_0+\epsilon_1\ell+\epsilon_2\ell^2\) are distinct, it follows that \(c_{j,\epsilon}=0\) for every \(j\) and \(\epsilon\). Thus every vector coefficient of \(F_w(T)\) has all its coordinates equal to zero, and hence
\[\pi\!\left( \bigl( E_+^{\epsilon_0}\otimes E_+^{\epsilon_1}\otimes E_+^{\epsilon_2} \bigr)w \right)=0\] for all \(\epsilon_0,\epsilon_1,\epsilon_2\in\{0,1\}\). Since \(w\in W'\) is arbitrary, each of these coefficient operators preserves \(W'\). In particular, \(W'\) is stable under
\[E_+\otimes I\otimes I,\quad I\otimes E_+\otimes I,\quad I\otimes I\otimes E_+.\]

Applying the same argument to the matrices \(I+tE_-\) with \(t\in\mathbb F_{\ell^3}\) shows that \(W'\) is also stable under
\[E_-\otimes I\otimes I,\quad I\otimes E_-\otimes I,\quad I\otimes I\otimes E_-.\]
As in the proof for \(\rho_1\), these six operators generate \(\operatorname{End}_k(W)\) as a \(k\)-algebra. Therefore \(W'\) is stable under every endomorphism of \(W\), so \(W'=0\) or \(W'=W\). Therefore \(\rho_2\) is irreducible.
\end{proof}

\begin{prop}\label{prop:bhps-orbit-growth}
Let \(\bar\xi\) be a geometric generic point of \(U_0\), and let \(\Gamma_\ell^{\operatorname{geom}}\subseteq\Sp_8(\mathbb F_\ell)\) be the geometric mod-\(\ell\) monodromy group of \(\mathcal J\to U_0\). Then \(\End(\mathcal J_{\bar\xi})=\mathbb Z.\) For all sufficiently large primes \(\ell\) and every nonzero proper subspace \(V\subset\mathcal J_{\bar\xi}[\ell]\), one has \(|\Gamma_\ell^{\operatorname{geom}}\cdot V|\ge\ell.\) Moreover, \(|\Gamma_\ell^{\operatorname{geom}}|\ll\ell^9,\) where the implicit constant depends only on this fixed family. In particular, \(\Gamma_\ell^{\operatorname{geom}}\ne\Sp_8(\mathbb F_\ell)\) for all sufficiently large \(\ell\).
\end{prop}

\begin{proof}
The generic endomorphism assertion follows from the Mumford-type identification in \cite[\S2.1 and Theorem~4.15]{BHPS2025}. We use the corresponding arithmetic description to determine a subgroup of the geometric mod-\(\ell\) monodromy group.

Let \(F:=\mathbb Q(\zeta_7+\zeta_7^{-1})\), where \(\zeta_7\) is a primitive seventh root of unity, and let \(\mathcal O_F\) be its ring of integers. Let \(D/F\) be the quaternion algebra ramified at exactly two real places and at no finite place. Choose a maximal order \(\mathcal O\subset D\), and let
\[\mathcal O^1:= \{u\in\mathcal O^\times:\operatorname{nrd}(u)=1\},\]
where \(\operatorname{nrd}:D\to F\) denotes the reduced norm. For this quaternion algebra, the corestriction \(\operatorname{Cor}_{F/\mathbb Q}(D)\) is split \cite[\S2.3]{BHPS2025}. Fixing an isomorphism with \(\operatorname{Mat}_8(\mathbb Q)\) gives the corestriction representation
\[\tau:\operatorname{Res}_{F/\mathbb Q}\operatorname{SL}_1(D) \rightarrow\Sp_8,\]
where \(\operatorname{SL}_1(D)\) is the algebraic group of elements of reduced norm one. By the corestriction construction \cite[\S2.1]{BHPS2025}, after base change to \(\overline{\mathbb Q}\) and choosing splittings of the three conjugates of \(D\), the representation \(\tau\) is given by
\[\operatorname{SL}_2(\overline{\mathbb Q})^3 \longrightarrow \operatorname{GL}\bigl((\overline{\mathbb Q}^{\,2})^{\otimes3}\bigr), \quad (g_1,g_2,g_3)\longmapsto g_1\otimes g_2\otimes g_3.\]

Choose \(t_*\in U_0(\mathbb C)\), let \(\Lambda:=H_1(\mathcal J_{t_*}(\mathbb C),\mathbb Z)\) and let \(\Gamma_{\mathrm B}\subseteq\operatorname{GL}(\Lambda)\) be the Betti monodromy group. The Shimura uniformization and corestriction description in \cite[\S\S2.1, 2.3, Proposition~4.14 and the proof of Theorem~4.15]{BHPS2025} imply that after a fixed rational identification, \(\Gamma_{\mathrm B}\) is commensurable with \(\tau(\mathcal O^1)\). Indeed, on a neat arithmetic cover the rational homology local system is obtained by descending \(\mathfrak h\times\mathbb Q^8\) with action \(\gamma(z,v)=(\gamma z,\tau(\gamma)v)\), where \(\mathfrak h\) is the complex upper half plane. Passing to this cover changes the monodromy by finite index, while homology lattices in the same rational local system are commensurable. Thus there is a finite-index subgroup \(\Delta\subseteq\mathcal O^1\) such that \(\tau(\Delta)\subseteq\Gamma_{\mathrm B}\) has finite index. Let \(m_0:=[\mathcal O^1:\Delta]\) and \(m_1:=[\Gamma_{\mathrm B}:\tau(\Delta)].\) The integers \(m_0,m_1\) are independent of \(\ell\).

Since \(D\) splits at one real place, strong approximation \cite[Corollary~28.5.14]{Voight2021} gives a surjection
\[\mathcal O^1\rightarrow Q_\ell:=(\mathcal O/\ell\mathcal O)^1,\]
where the superscript \(1\) denotes the units of reduced norm one modulo \(\ell\). For all sufficiently large primes \(\ell\), we have
\[Q_\ell\simeq \prod_{\mathfrak p\mid\ell\mathcal O_F} \operatorname{SL}_2(k_{\mathfrak p}), \quad k_{\mathfrak p}:=\mathcal O_F/\mathfrak p,\]
where the product runs over the prime ideals of \(\mathcal O_F\) above \(\ell\). The image \(\overline\Delta_\ell\) of \(\Delta\) in \(Q_\ell\) has index at most \(m_0\). If \(\ell>m_0\), this image is all of \(Q_\ell\): the coset action has degree less than \(\ell\), whereas \(Q_\ell\) is generated by its elementary unipotents of order \(\ell\). These generators act trivially on the coset space, which is transitive and hence consists of one point.

After excluding finitely many primes, the coefficients defining \(\tau\) and the matrices of the fixed rational change of basis and its inverse are all \(\ell\)-integral. Reduction then gives a representation
\[\tau_\ell: Q_\ell\longrightarrow \operatorname{GL}_{\mathbb F_\ell}(\Lambda/\ell\Lambda).\]
Write \(\operatorname{red}_\ell\) for reduction of automorphisms of \(\Lambda\) modulo \(\ell\). For every \(u\in\Delta\), we have
\[\operatorname{red}_\ell(\tau(u))=\tau_\ell(u\bmod\ell\mathcal O).\]
Since the image of \(\Delta\) in \(Q_\ell\) is all of \(Q_\ell\), it follows that
\[G_\ell:=\tau_\ell(Q_\ell) = \operatorname{red}_\ell(\tau(\Delta)).\]

The natural isomorphism \(\mathcal J_{t_*}(\mathbb C)[\ell] \xrightarrow{\sim}\Lambda/\ell\Lambda\) is equivariant for the action of \(\pi_1(U_0(\mathbb C),t_*)\). Under this isomorphism, the topological monodromy on the \(\ell\)-torsion is therefore \(\operatorname{red}_\ell(\Gamma_{\mathrm B})\). Betti--\'etale comparison \cite[Expos\'e~XII, Corollaire~5.2]{SGA1} together with invariance of the \'etale fundamental group under extension of algebraically closed fields in characteristic zero \cite[Remark~3.3(c)]{Milne2013Etale} gives \(\Gamma_\ell^{\operatorname{geom}} =\operatorname{red}_\ell(\Gamma_{\mathrm B}),\) where both groups are viewed as acting on \(\Lambda/\ell\Lambda\).

Since \(\tau(\Delta)\subseteq\Gamma_{\mathrm B}\), we have \(G_\ell\subseteq\Gamma_\ell^{\operatorname{geom}}.\) Moreover, since \([\Gamma_{\mathrm B}:\tau(\Delta)]=m_1\), reducing a set of \(m_1\) coset representatives shows that \([\Gamma_\ell^{\operatorname{geom}}:G_\ell]\le m_1.\) Consequently,
\[|\Gamma_\ell^{\operatorname{geom}}| \le m_1|G_\ell| \le m_1|Q_\ell|.\]

Since \(F/\mathbb Q\) is cyclic of degree 3, every unramified prime is either split or inert. In the split case, \(Q_\ell\simeq\operatorname{SL}_2(\mathbb F_\ell)^3\); in the inert case, \(Q_\ell\simeq\operatorname{SL}_2(\mathbb F_{\ell^3})\). After extending scalars to \(\overline{\mathbb F}_\ell\), the representation \(\tau_\ell\) is isomorphic to \(\rho_1\) in the split case and to \(\rho_2\) in the inert case, where \(\rho_1,\rho_2\) are defined in Lemma~\ref{lem:mumford-residual-tensor}. Thus the action of \(G_\ell\) is absolutely irreducible. The images of the elementary unipotents generate \(G_\ell\), and their nonidentity images have order \(\ell\). Lemma~\ref{lem:ell-generated-orbits} gives
\[|\Gamma_\ell^{\operatorname{geom}}\cdot V| \ge |G_\ell\cdot V|\ge\ell.\]
Finally,
\[
|Q_\ell|=
\begin{cases}
\ell^3(\ell^2-1)^3,&\ell\text{ split in }F,\\
\ell^3(\ell^6-1),&\ell\text{ inert in }F.
\end{cases}
\]
Hence \(|\Gamma_\ell^{\operatorname{geom}}|\ll\ell^9\), whereas \(|\Sp_8(\mathbb F_\ell)| =\ell^{16}\prod_{i=1}^4(\ell^{2i}-1)\asymp\ell^{36}.\)
\end{proof}

Now we apply Theorem~\ref{thm:polynomial-orbit-growth}. By Lemma~\ref{lem:bhps-birational-parameter}, the extended moduli map is the normalization parametrization \(\eta:\mathbb P^1_{\mathbb Q}\to\overline X\). After restricting to a nonempty open subset \(U\subseteq U_0\), it identifies \(U\) with an open subset of \(X\). We replace \(X\) by this open subset and regard \(\mathcal J\) as an abelian scheme over \(X\).  The map \(\pi_1(U_{\overline{\mathbb Q}}) \rightarrow\pi_1(U_{0,\overline{\mathbb Q}})\)
is surjective, so this restriction does not change the geometric monodromy images. The same holds after extension to any number field \(K\).

By Proposition~\ref{prop:bhps-orbit-growth}, the endomorphism and orbit-growth hypotheses of Theorem~\ref{thm:polynomial-orbit-growth} are satisfied, with \(\alpha_{a,r}=1\) for every \(1\le a\le3\) and \(0\le r\le\min(a,4-a)\). Therefore we have \(\alpha_a^{\min}=1\) and \[\beta_X=\max_{1\le a\le3}8a(4-a)=32.\] So Theorem~\ref{thm:polynomial-orbit-growth} gives
\[
|S_X(B)|\ll_{K,X}(\log B)^{\kappa_0+32},\qquad
|S_\eta(B)|\ll_{K,X}(\log B)^{\kappa_0+32}
\]
for all sufficiently large \(B\), with the threshold for \(S_\eta(B)\) allowed to depend on \(\eta\). Taking \(\kappa_0=12\), the exponent in both estimates is \(44=12+8\lfloor4^2/4\rfloor\).
Since the geometric monodromy is a proper subgroup of \(\Sp_8(\mathbb F_\ell)\) for all sufficiently large \(\ell\), this family is not covered by Corollary~\ref{cor:full-monodromy}.
\section*{Compliance with ethical standards}
\subsection*{Conflicts of interest}
  The author states that there are no conflicts of interest.

\subsection*{Data availability} Data sharing is not applicable to this article as no data sets were generated or analyzed.

% Review markup: color only the corrected bibliography entries.
% Remove this local wrapper to typeset the bibliography entirely in black.

\bibliography{references}
\end{sloppypar}
\end{document}